\documentclass[10pt]{amsart}
\usepackage{graphicx,amssymb}
\usepackage{a4wide,color}
\usepackage[matrix,arrow,curve]{xy}

\definecolor{light}{gray}{.75}

\theoremstyle{plain}
\newtheorem{theorem}{Theorem}[section]
\newtheorem{lemma}[theorem]{Lemma}
\newtheorem{corollary}[theorem]{Corollary}
\newtheorem{proposition}[theorem]{Proposition}

\theoremstyle{definition}
\newtheorem{definition}[theorem]{Definition}
\newtheorem{example}[theorem]{Example}

\newenvironment{claim}[1]
{\par\addvspace{\medskipamount}\noindent\emph{#1.}\em\ \ignorespaces}
{\normalfont\par\addvspace{\medskipamount}}

\def\arA{\mathbf A}
\def\arB{\mathbf B}
\def\arC{\mathbf C}
\def\arD{\mathbf D}
\def\lk{\operatorname{lk}}

\def\d{\mathbf d}

\def\Z{\mathbf Z}
\def\n{\mathbf n}

\def\C{\mathcal C}
\def\D{\mathcal D}

\def\ext{{\operatorname{ext}}}
\def\R{\operatorname{\mathcal R}}
\def\Ext{\operatorname{Ext}}
\def\Rext{\R_\ext}

\def\myangle#1{\langle #1\rangle}

\def\case#1#2{\medskip\noindent\textbf{#1.\ \ #2}\hskip 1em\ignorespaces}

\begin{document}

\title
[Injectivity on the set of conjugacy classes]
{Injectivity on the set of conjugacy classes\\
of some monomorphisms between Artin groups}

\author{Eon-Kyung Lee and Sang-Jin Lee}

\address{Department of Mathematics, Sejong University,
Seoul, 143-747, Korea}
\email{eonkyung@sejong.ac.kr}

\address{Department of Mathematics, Konkuk University,
Seoul, 143-701, Korea}
\email{sangjin@konkuk.ac.kr}

\date{\today}

\begin{abstract}
There are well-known monomorphisms between
the Artin groups of finite type $\arA_n$, $\arB_n=\arC_n$
and affine type $\tilde \arA_{n-1}$, $\tilde\arC_{n-1}$.
The Artin group $A(\arA_n)$ is isomorphic to
the $(n+1)$-strand braid group $B_{n+1}$,
and the other three Artin groups are isomorphic to
some subgroups of $B_{n+1}$.
The inclusions between these subgroups yield monomorphisms
$A(\arB_n)\to A(\arA_n)$,
$A(\tilde \arA_{n-1})\to A(\arB_n)$ and
$A(\tilde \arC_{n-1})\to A(\arB_n)$.
There are another type of monomorphisms
$A(\arB_d)\to A(\arA_{md-1})$, $A(\arB_d)\to A(\arB_{md})$
and $A(\arB_d)\to A(\arA_{md})$
which are induced by isomorphisms between
Artin groups of type $\arB$ and centralizers of periodic braids.

In this paper, we show that the monomorphisms
$A(\arB_d)\to A(\arA_{md-1})$, $A(\arB_d)\to A(\arB_{md})$
and $A(\arB_d)\to A(\arA_{md})$
induce injective functions on the set of conjugacy classes,
and that none of the monomorphisms $A(\arB_n)\to A(\arA_n)$,
$A(\tilde \arA_{n-1})\to A(\arB_n)$ and
$A(\tilde \arC_{n-1})\to A(\arB_n)$ does so.

\medskip\noindent
\emph{Key words:} Artin group, braid group, conjugacy class,
Nielsen-Thurston classification\\
\emph{MSC:} 20F36, 20F10
\end{abstract}

\maketitle


\section{Introduction}

Let $M$ be a symmetric $n\times n$ matrix with entries
$m_{ij}\in\{2,3,\ldots,\infty\}$ for $i\ne j$,
and $m_{ii}=1$ for $1\le i\le n$.
The \emph{Artin group} of type $M$ is defined by the presentation
$$
A(M)=\langle s_1,\ldots,s_n\mid \underbrace{s_is_js_i\cdots}_{m_{ij}}
=\underbrace{s_js_is_j\cdots}_{m_{ij}}
\quad\mbox{for all $i\ne j$, $m_{ij}\ne\infty$}\rangle.
$$
The \emph{Coxeter group} $W(M)$ of type $M$ is the quotient
of $A(M)$ by the relation $s_i^2=1$.
The Artin group $A(M)$ is said to be \emph{of finite type}
if $W(M)$ is finite,
and \emph{of affine type}
if $W(M)$ acts as a proper, cocompact group of isometries on some
Euclidean space with the generators $s_1,\ldots,s_n$ acting
as affine reflections.
Artin groups are usually described by \emph{Coxeter graphs},
whose vertices are numbered $1,\ldots,n$
and which has an edge labelled $m_{ij}$
between the vertices $i$ and $j$
whenever $m_{ij}\ge 3$ or $m_{ij}=\infty$.
The label 3 is usually suppressed.

There are well-known monomorphisms
between the Artin groups of finite type $\arA_n$, $\arB_n=\arC_n$,
and affine type $\tilde \arA_{n-1}$, $\tilde \arC_{n-1}$.
The Coxeter graphs associated to these Artin groups
are as follows.

\smallskip
\begin{center}
\begin{tabular}{rlcrl}
\raisebox{.9em}{$\arA_n$} &
\includegraphics[scale=.4]{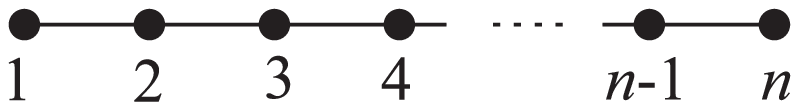} &\mbox{}\qquad\mbox{}&
\raisebox{.9em}{$\tilde \arA_{n-1}$} &
\includegraphics[scale=.4]{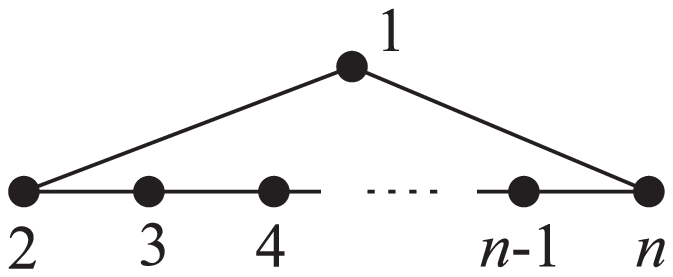}\\
\raisebox{.9em}{$\arB_n$} &
\includegraphics[scale=.4]{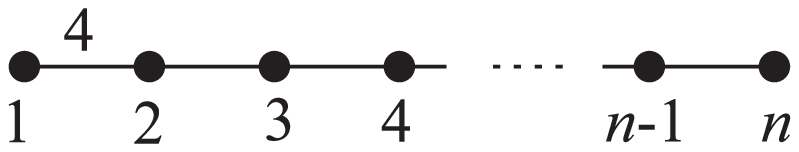} &&
\raisebox{.9em}{$\tilde \arC_{n-1}$} &
\includegraphics[scale=.4]{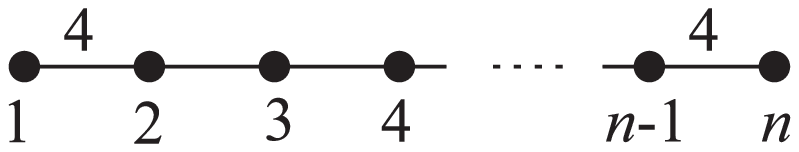}
\end{tabular}
\end{center}
\smallskip

In this paper, we investigate whether these monomorphisms
induce injective functions on the set of conjugacy classes or not.
Before stating our results, let us review briefly
braid groups and monomorphisms between these Artin groups.

\smallskip

Let $D^2=\{z\in\mathbb{C}\mid |z|\le n+1\}$,
and let $D_n$ be the $n$-punctured disk $D^2\setminus\{1,\ldots,n\}$.
The braid group $B_n$ is defined as the group of self-diffeomorphisms of $D_{n}$
that fix the boundary pointwise,
modulo isotopy relative to the boundary.
Equivalently, $n$-braids can be defined as isotopy classes
of collections of pairwise disjoint $n$ strands
$l=l_1\cup\cdots \cup l_n$ in $D^2\times[0,1]$ such that
$l\cap (D^2\times\{t\})$ consists of $n$ points for each $t\in[0,1]$,
and $l\cap (D^2\times\{0,1\})=\{1,\ldots,n\}\times\{0,1\}$.
The admissible isotopies lie in the interior of $D^2\times [0,1]$.
The $n$-braid group $B_n$ has the well-known Artin presentation~\cite{Art25}:
$$
B_n=\left\langle\sigma_1,\ldots,\sigma_{n-1}\biggm|
\begin{array}{ll}
\sigma_i\sigma_j=\sigma_j\sigma_i & \mbox{if } |i-j| \ge 2 \\
\sigma_i\sigma_{j}\sigma_i=\sigma_{j}\sigma_i\sigma_{j}
& \mbox{if } |i-j| = 1
\end{array}
\right\rangle.
$$

\begin{definition}
For $\alpha\in B_n$, let $\pi_\alpha$ denote
its induced permutation on $\{1,\ldots,n\}$.
If $\pi_\alpha(i)=i$, we say that $\alpha$ is \emph{$i$-pure},
or the $i$-th strand of $\alpha$ is \emph{pure}.
For $P\subset\{1,\ldots,n\}$,
we say that $\alpha$ is \emph{$P$-pure}
if $\alpha$ is $i$-pure for each $i\in P$.
\end{definition}

\begin{definition}\label{def:linking}
Let $B_{n,1}$ denote the subgroup of $B_n$ consisting of 1-pure braids.
Let $\lk:B_{n,1}\to \mathbb Z$ be the homomorphism measuring
the linking number of the first strand with the other strands:
$B_{n,1}$ is generated by $\sigma_1^2,\sigma_2,\ldots,\sigma_{n-1}$,
and $\lk$ is defined by $\lk(\sigma_1^2)=1$ and
$\lk(\sigma_i)=0$ for $i\ge 2$.
A braid $\alpha$ is said to be \emph{1-unlinked}
if it is 1-pure and $\lk(\alpha)=0$.
Let $\nu:B_{n,1}\to B_{n-1}$ be the homomorphism
deleting the first strand.
\end{definition}

For example, the braid in Figure~\ref{fig:p-pure}
is $\{1,4,5\}$-pure and 1-unlinked.
Note that $\{1,\ldots,n\}$-pure braids are
nothing more than pure braids in the usual sense,
and that $\lk$ is a conjugacy invariant of 1-pure braids
because
$\lk(\beta\alpha\beta^{-1})
=\lk(\beta)+\lk(\alpha)-\lk(\beta)
=\lk(\alpha)$ for $\alpha,\beta\in B_{n,1}$.

\begin{figure}
\includegraphics[scale=.8]{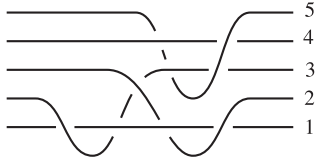}
\caption{This braid is $\{1,4,5\}$-pure  and 1-unlinked.}
\label{fig:p-pure}
\end{figure}

\medskip

From the presentation of $B_n$, it is obvious that
the Artin group $A(\arA_n)$ is isomorphic to $B_{n+1}$.
The other Artin groups $A(\arB_n)$, $A(\tilde \arA_{n-1})$
and $A(\tilde \arC_{n-1})$ are known to be isomorphic
to subgroups of $B_{n+1}$ as follows~\cite{All02, CC05, BM07}.
\begin{eqnarray*}
A(\arB_n) &\simeq&
\{\alpha\in B_{n+1}\mid \mbox{$\alpha$ is 1-pure} \}
= B_{n+1,1} \subset B_{n+1};\\
A(\tilde \arA_{n-1}) &\simeq&
\{\alpha\in B_{n+1}\mid \mbox{$\alpha$ is 1-unlinked} \}
\subset B_{n+1,1};\\
A(\tilde \arC_{n-1}) &\simeq&
\{\alpha\in B_{n+1}\mid \mbox{$\alpha$ is $\{1,n+1\}$-pure} \}
\subset B_{n+1,1}.
\end{eqnarray*}
From now on, we identify $A(\arA_n)$ with $B_{n+1}$, and
the Artin groups $A(\arB_n)$, $A(\tilde \arA_{n-1})$
and $A(\tilde \arC_{n-1})$ with the corresponding subgroups of $B_{n+1}$
induced by the above isomorphisms.
Then the inclusions between the subgroups of $B_{n+1}$ yield monomorphisms
$\psi_1:A(\arB_n)\to A(\arA_n)$,
$\psi_2:A(\tilde \arA_{n-1}) \to A(\arB_n)$ and
$\psi_3:A(\tilde \arC_{n-1}) \to A(\arB_n)$.
See Figure~\ref{fig:inj}.

\begin{figure}
$
\xymatrix{
& A(\arA_n)\\
& A(\arB_n) \ar[u]_{\psi_1}\\
A(\tilde \arA_{n-1}) \ar[ur]^{\psi_2}
  && A(\tilde \arC_{n-1}) \ar[ul]_{\psi_3}
}
\xymatrix{
A(\arA_{md-1})   && A(\arA_{md}) \\
& A(\arB_{md}) \ar[ul]_{\nu}\ar[ur]^{\psi_1}\\
& A(\arB_d)    \ar[u]_{\psi_4}\ar[luu]^{\psi_5}\ar[ruu]_{\psi_6}
}
$
\caption{$\psi_1,\ldots,\psi_6$ are monomorphisms between Artin groups.}
\label{fig:inj}
\end{figure}
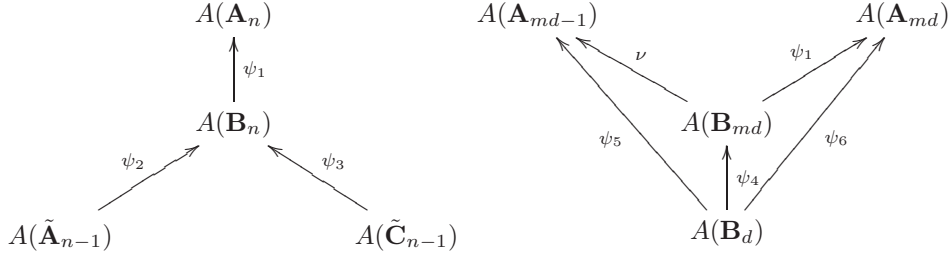

\medskip

There are another type of monomorphisms
which can be described by isomorphisms
between Artin groups of type $\arB$ and
centralizers of periodic braids.
The well-known Nielsen-Thurston classification
of mapping classes of surfaces (possibly with punctures)
with negative Euler characteristic into periodic,
reducible and pseudo-Anosov ones~\cite{Thu88}
yields an analogous classification of braids.
An $n$-braid $\alpha$ is said to be \emph{periodic}
if some power of it is a power of $\Delta^2$,
where $\Delta=\sigma_1(\sigma_2\sigma_1)\cdots
(\sigma_{n-1}\cdots\sigma_1)$;
\emph{reducible} if there is an essential curve
system in $D_n$ whose isotopy class is invariant
under the action of $\alpha$;
\emph{pseudo-Anosov} if there is a pseudo-Anosov diffeomorphism $f$
defined on the interior of $D_n$
representing $\alpha$ modulo $\Delta^2$,
that is, there exist a pair of transverse measured foliations
$(F^s,\mu^s)$ and $(F^u,\mu^u)$ and a real $\lambda>1$,
called \emph{dilatation}, such that
$f(F^s,\mu^s)=(F^s,\lambda^{-1}\mu^s)$ and
$f(F^u,\mu^u)=(F^u,\lambda\mu^u)$.
Every braid belongs to one of the following
three mutually disjoint classes:
periodic;
pseudo-Anosov;
non-periodic and reducible.
We say that two braids are \emph{of the same Nielsen-Thurston type}
if both of them belong to the same class.

Let $\delta=\sigma_{n-1}\cdots\sigma_1$ and $\epsilon=\delta\sigma_1$,
then $\delta^n = \Delta^2 = \epsilon^{n-1}$.
If we need to specify the braid index $n$,
we will write $\delta=\delta_{(n)}$,
$\epsilon = \epsilon_{(n)}$ and $\Delta = \Delta_{(n)}$.
The braids $\delta$ and $\epsilon$ are induced
by rigid rotations of the punctured disk as in Figure~\ref{fig:circ}
when the punctures are at the origin or evenly distributed
on a round circle centered at the origin.
Due to L.E.J.~Brouwer~\cite{Bro19}, B.~de K\'er\'ekj\'art\'o~\cite{Ker19} and
S.~Eilenberg~\cite{Eil34} (see also~\cite{CK94}),
it is known that an $n$-braid is periodic if and only if
it is conjugate to a power of $\delta$ or $\epsilon$.

\begin{figure}
\begin{tabular}{ccc}
\includegraphics[scale=.65]{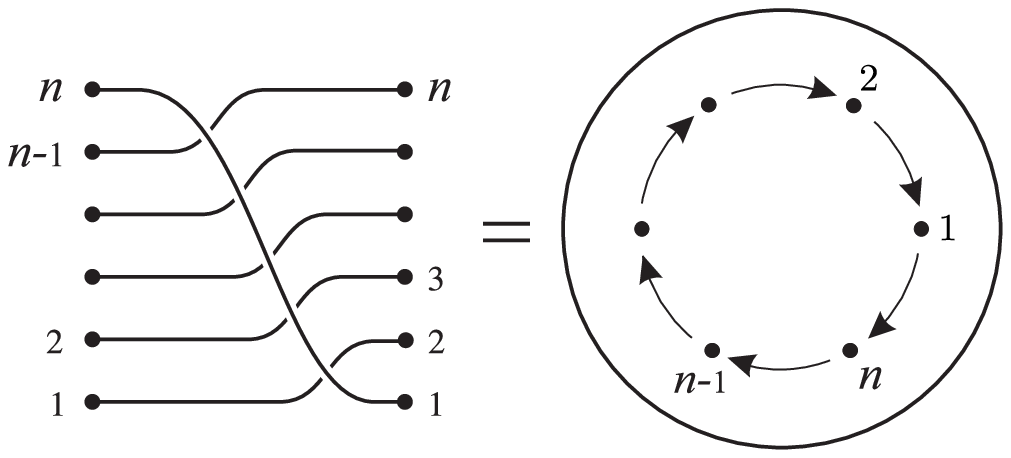}
& &
\includegraphics[scale=.65]{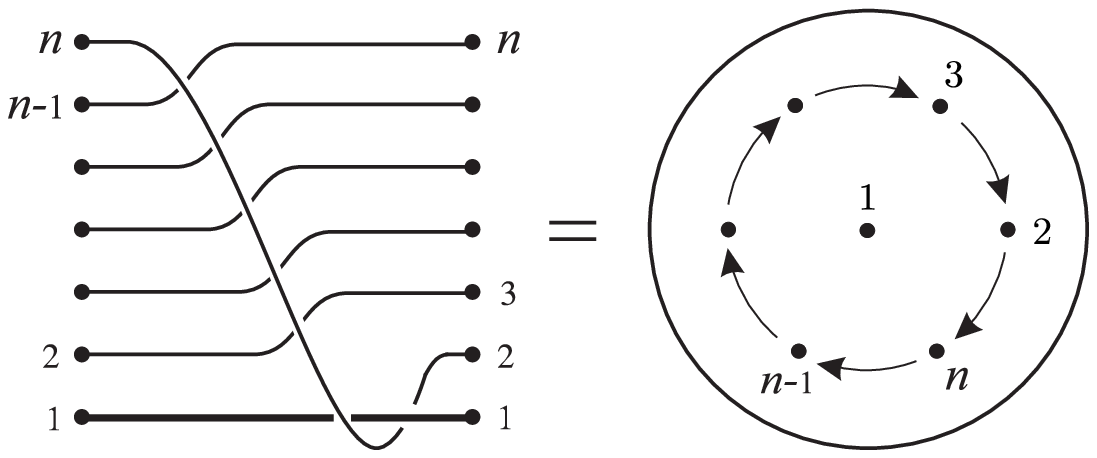}\\
(a) $\delta_{(n)} = \sigma_{n-1}\sigma_{n-2}\cdots\sigma_1 \in B_{n}$
&& (b) $\epsilon_{(n)} = \delta_{(n)}\sigma_1\in B_{n}$
\end{tabular}
\caption{
The braids $\delta$ and $\epsilon$
are represented by rigid rotations of punctured disks.}
\label{fig:circ}
\end{figure}

For a group $G$ and an element $g\in G$, let $Z_G(g)$ denote
the centralizer of $g$ in $G$, that is, $Z_G(g)=\{h\in G:gh=hg\}$.
If the group $G$ is clear from the context, we write simply $Z(g)$.
The following theorem is a consequence of
the result of D.~Bessis, F.~Digne and J.~Michel~\cite{BDM02}
on irreducible complex reflection groups.
See also~\cite{GW04}.

\begin{theorem}[\cite{BDM02}]
\label{thm:cent}
Let $\omega\in B_n$ be conjugate to $\delta^k$ (resp.\ $\epsilon^k$),
where $k$ is not a multiple of\/ $n$ (resp.\ $n-1$).
Then the centralizer $Z(\omega)$ is isomorphic to
the Artin group $A(\arB_d)$,
where $d=\gcd(k,n)$ (resp.\ $d=\gcd(k,n-1)$).
\end{theorem}

Let $m\ge 2$ and $d\ge 1$.
The centralizer $Z(\epsilon_{(md+1)}^d)$ consists of 1-pure braids
because the first strand of $\epsilon_{(md+1)}^d$ is the only pure strand.
By theorem~\ref{thm:cent}, we have
\begin{eqnarray*}
A(\arB_d) &\simeq & Z(\epsilon_{(md+1)}^d)
\subset B_{md+1,1}=A(\arB_{md})
\subset B_{md+1}=A(\arA_{md});\\
A(\arB_d) &\simeq & Z(\delta_{(md)}^d)\subset B_{md}=A(\arA_{md-1}).
\end{eqnarray*}
Compositions of the above isomorphisms and inclusions yield monomorphisms
$\psi_4:A(\arB_d)\to A(\arB_{md})$,
$\psi_5:A(\arB_d)\to A(\arA_{md-1})$ and
$\psi_6:A(\arB_d)\to A(\arA_{md})$.
The monomorphisms $\psi_4$ and $\psi_5$ in~\cite{BDM02}
and~\cite{GW04} are such that
$\psi_4(\epsilon_{(d+1)})=\epsilon_{(md+1)}$,
$\psi_5(\epsilon_{(d+1)})=\delta_{(md)}$ and
$\psi_5=\nu\circ\psi_4$,
where $\nu: B_{md+1,1}\to B_{md}$ is the homomorphism deleting the first strand
(see Definition~\ref{def:linking}).
By definition, $\psi_6=\psi_1\circ \psi_4$.
See Figure~\ref{fig:inj}.
In~\cite{BGG06} J.~Birman, V.~Gebhardt and J.~Gonz\'alez-Meneses
described explicitly in terms of standard generators
the monomorphism $\psi_5: A(\arB_d)\to A(\arA_{md-1})$ when $m=2$.

\medskip

The following theorem is the main result of this paper.

\begin{theorem}
\label{thm:conj}
Let $\omega$ be a periodic $n$-braid,
and let $\alpha,\beta\in Z(\omega)$.
If $\alpha$ and $\beta$ are conjugate in $B_n$,
then they are conjugate in $Z(\omega)$.
\end{theorem}

We prove the above theorem,
relying on the Nielsen-Thurston classification of braids.
In the case of pseudo-Anosov braids, we obtain a stronger result:
if $\alpha,\beta\in Z(\omega)$ are pseudo-Anosov and conjugate in $B_n$,
then any conjugating element from $\alpha$ to $\beta$
belongs to $Z(\omega)$ (see Proposition~\ref{prop:pAconj}).
The case of periodic braids is easy to prove
due to the uniqueness of roots up to conjugacy
for Artin groups of type $\arB$~\cite{LL07},
hence most part of this paper is devoted to the case of reducible braids.

\medskip

Using the characterization of $Z(\omega)$ as the Artin group
of type $\arB$, we obtain the following corollary
which is equivalent to the above theorem
(see Lemma~\ref{lem:equiv} for the equivalence).

\begin{corollary}\label{thm:main2}
For $m\ge 2$ and $d\ge 1$, the monomorphisms
$\psi_5:A(\arB_d)\to A(\arA_{md-1})$ and
$\psi_6:A(\arB_d)\to A(\arA_{md})$
induce injective functions on the set of conjugacy classes.
\end{corollary}

The above corollary implies that $\psi_4$ is also injective
on the set of conjugacy classes
because $\psi_5=\nu\circ\psi_4$ and $\psi_6=\psi_1\circ \psi_4$.

\medskip

We also investigate the injectivity for
the well-known monomorphisms of $A(\tilde \arA_{n-1})$
and $A(\tilde \arC_{n-1})$ into $A(\arB_n)$, and a further monomorphism
of $A(\arB_n)$ into $A(\arA_n)$, and have the following.

\begin{theorem}\label{thm:main1}
For $n\ge 3$, none of the monomorphisms
$\psi_1:A(\arB_n)\to A(\arA_n)$,
$\psi_2:A(\tilde \arA_{n-1}) \to A(\arB_n)$ and
$\psi_3:A(\tilde \arC_{n-1}) \to A(\arB_n)$
induces an injective function on the set of conjugacy classes.
\end{theorem}

For the proof of the above theorem, we present three examples
which show that the monomorphisms $\psi_1$, $\psi_2$ and $\psi_3$
are not injective on the set of conjugacy classes.
For instance, the first example is a pair of elements in $B_{n+1,1}$
that are conjugate in $B_{n+1}$, but not conjugate in $B_{n+1,1}$.

\medskip

We close this section with a remark on a monomorphism
from $A(\arA_{n-1})$ to $A(\arD_n)$.
The Coxeter graph of type $\arD_n$ is as follows.

\smallskip
\begin{center}\begin{tabular}{cc}
\raisebox{.6em}{$\arD_n$} &
\includegraphics[scale=.5]{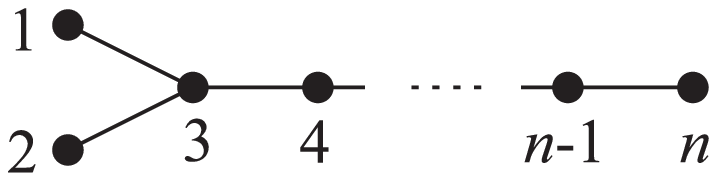}
\end{tabular}
\end{center}
\smallskip

\noindent
Let $\sigma_1,\ldots,\sigma_{n-1}$ and $t_1,\ldots,t_n$ be
the standard generators of $A(\arA_{n-1})$ and $A(\arD_n)$.
Let $\iota: A(\arA_{n-1})\to A(\arD_n)$ and $\pi:A(\arD_n)\to A(\arA_{n-1})$
be the homomorphisms defined by
$\iota(\sigma_i)=t_{i+1}$ for $1\le i\le n-1$;
$\pi(t_1)=\sigma_1$; $\pi(t_i)=\sigma_{i-1}$ for $i\ge 2$.
Because $\pi\circ\iota$ is the identity,
$\pi$ is an epimorphism and
$\iota$ is a monomorphism,
further, $\iota$ is
injective on the set of conjugacy classes.
It is known by J.~Crisp and L.~Paris
that $A(\arD_n)$ is isomorphic
to a semidirect product $F_{n-1}\rtimes A(\arA_{n-1})$,
where $F_{n-1}$ is a free group of rank $n-1$~\cite{CP05}.

\section{Preliminaries}
In this section, we review the uniqueness of roots up to
conjugacy in some Artin groups, and some tools useful in handling reducible braids
such as canonical reduction system and standard reduction system.

\subsection{Periodic braids}
The center of $B_n$ is the infinite cyclic group generated by $\Delta^2$.
An $n$-braid is periodic if and only if it is of finite order
in the central quotient $B_n/\langle \Delta^2\rangle$.
Note that $\delta$ generates a cyclic group of order $n$
in $B_n/\langle \Delta^2\rangle$,
hence $\delta^a$ and $\delta^b$ generate
the same group in $B_n/\langle \Delta^2\rangle$
if and only if $\gcd(a,n)=\gcd(b,n)$.
In fact, we have the following.

\begin{lemma}\label{lem:order}
Let $a$ and $b$ be integers.
In $B_n$, the following hold.
\begin{itemize}
\item[(i)]
$Z(\delta^a)=Z(\delta^b)$ if and only if\/ $\gcd(a,n)=\gcd(b,n)$.
\item[(ii)]
$Z(\epsilon^a)=Z(\epsilon^b)$ if and only if\/ $\gcd(a,n-1)=\gcd(b,n-1)$.
\item[(iii)]
$Z(\delta^a)=Z(\delta^{\gcd(a,n)})$ and
$Z(\epsilon^a)=Z(\epsilon^{\gcd(a,n-1)})$.
\end{itemize}
\end{lemma}
In the above, (iii) is a direct consequence of (i) and (ii).

The central quotient $B_n/\langle \Delta^2\rangle$ is
the group of isotopy classes of
self-diffeomorphisms of $D_n$ without the condition of
fixing $\partial D_n$ pointwise on self-diffeomorphisms and isotopies.
If a diffeomorphism $f:D_n\to D_n$ does not fix
$\partial D_n$ pointwise,
it determines a braid modulo $\Delta^2$.

\subsection{Uniqueness of roots up to conjugacy in some Artin groups }

It is known that the $k$-th roots of a braid are unique up to conjugacy
by J.~Gonz\'alez-Meneses~\cite{Gon03}, and this is
generalized to the Artin groups $A(\arB_n)= A(\arC_n)$, $A(\tilde \arA_n)$
and $A(\tilde \arC_n)$~\cite{LL07}.

\begin{theorem}[J.~Gonz\'alez-Meneses~\cite{Gon03}]
\label{thm:gon}
Let $\alpha$ and $\beta$ be $n$-braids such that $\alpha^k=\beta^k$
for some nonzero integer $k$. Then $\alpha$ and $\beta$ are conjugate
in $B_n$.
\end{theorem}

\begin{theorem}[E.-K.~Lee and S.-J.~Lee~\cite{LL07}]\label{thm:artin}
Let\/ $G$ denote one of the Artin groups of finite type
$\mathbf A_n$, $\mathbf B_n=\mathbf C_n$ and
affine type $\tilde {\mathbf A}_{n-1}$, $\tilde {\mathbf C}_{n-1}$.
If $\alpha,\beta\in G$ are such that $\alpha^k=\beta^k$
for some nonzero integer $k$,
then $\alpha$ and $\beta$ are conjugate in $G$.
\end{theorem}

\begin{corollary}\label{cor:unique}
Let\/ $G$ denote one of the Artin groups of finite type
$\mathbf A_n$, $\mathbf B_n=\mathbf C_n$ and
affine type $\tilde {\mathbf A}_{n-1}$, $\tilde {\mathbf C}_{n-1}$.
Let $\alpha$ and $\beta$ be elements of\/ $G$,
and let $k$ be a nonzero integer.
Then, $\alpha$ and $\beta$ are conjugate in $G$ if and only if
so are $\alpha^k$ and $\beta^k$.
\end{corollary}

\subsection{Canonical reduction system of reducible braids}
Abusing notation, we use the same symbol for a curve and its isotopy class.
Hence, for curves $C_1$ and $C_2$,
$C_1=C_2$ means that $C_1$ and $C_2$ are isotopic,
unless stated otherwise explicitly.

A curve system $\C$ in $D_n$ means a finite collection
of disjoint simple closed curves in $D_n$.
It is said to be \emph{essential}\/
if each component is homotopic neither to a point
nor to a puncture nor to the boundary.
For an $n$-braid $\alpha$ and a curve system $\C$ in $D_n$,
let $\alpha*\C$ denote the left action of $\alpha$ on $\C$.
Recall that an $n$-braid $\alpha$ is reducible
if $\alpha*\C=\C$ for some essential curve system $\C$ in $D_n$,
called a \emph{reduction system} of $\alpha$.

For a reduction system $\C$ of an $n$-braid $\alpha$,
let $D_\C$ be the closure of $D_n\setminus N(\C)$ in
$D_n$, where $N(\C)$ is a regular neighborhood of $\C$.
The restriction of $\alpha$ induces a self-diffeomorphism on $D_\C$ that is
well defined up to isotopy. Due to J.~Birman, A.~Lubotzky and
J.~McCarthy~\cite{BLM83} and N.~V.~Ivanov~\cite{Iva92},
for any $n$-braid $\alpha$,
there is a unique
\emph{canonical reduction system} $\R(\alpha)$ with the following
properties.
\begin{enumerate}
\item[(i)]
$\R(\alpha^m)=\R(\alpha)$ for all $m\ne 0$.

\item[(ii)]
$\R(\beta\alpha\beta^{-1})=\beta*\R(\alpha)$ for all $\beta\in B_n$.

\item[(iii)]
The restriction of $\alpha$ to each component of $D_{\R(\alpha)}$ is
either periodic or pseudo-Anosov. A reduction system with this
property is said to be \emph{adequate}.

\item[(iv)]
If $\C$ is an adequate reduction system of $\alpha$,
then $\R(\alpha)\subset\C$.
\end{enumerate}

Note that a braid $\alpha$ is non-periodic and reducible if and only if
$\R(\alpha)\ne\emptyset$.
Let $\Rext(\alpha)$ denote the collection
of outermost components of $\R(\alpha)$.
Then $\Rext(\alpha)$ satisfies the properties (i) and (ii).

\subsection{Reducible braids with a standard reduction system}

Here we introduce some notions in~\cite{LL08}
that will be used in computations involving reducible braids.

\begin{definition}
An essential curve system in $D_n$ is said to be \emph{standard}\/
if each component is isotopic to a round circle
centered at the real axis as in Figure~\ref{fig:standard}~(a),
and \emph{unnested} if none of its components
encloses another component as in Figure~\ref{fig:standard}~(b).
Two curve systems $\C_1$ and $\C_2$ in $D_n$
are said to be \emph{of the same type}
if there is a diffeomorphism
$f:D_n\to D_n$ such that $f(\C_1)$ is isotopic to $\C_2$.
\end{definition}

\begin{figure}
\begin{tabular}{ccccc}
\includegraphics[scale=.5]{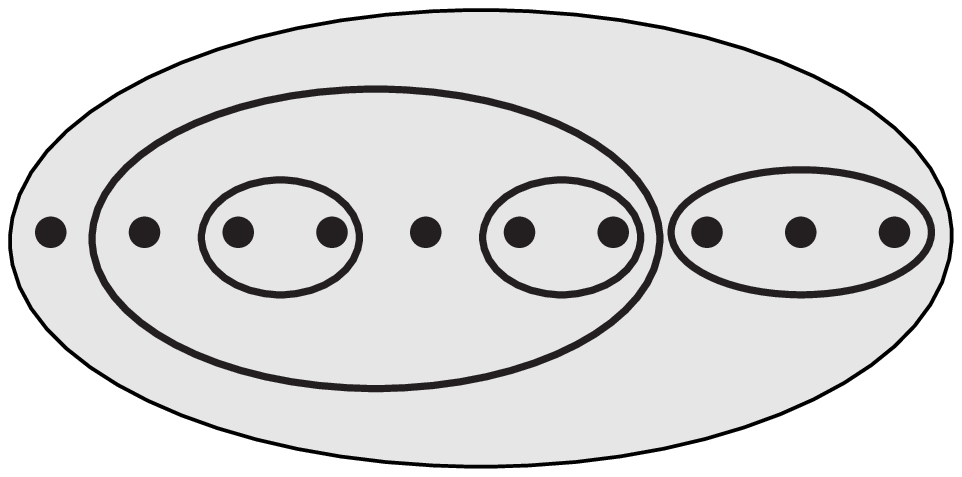}
&& \includegraphics[scale=.5]{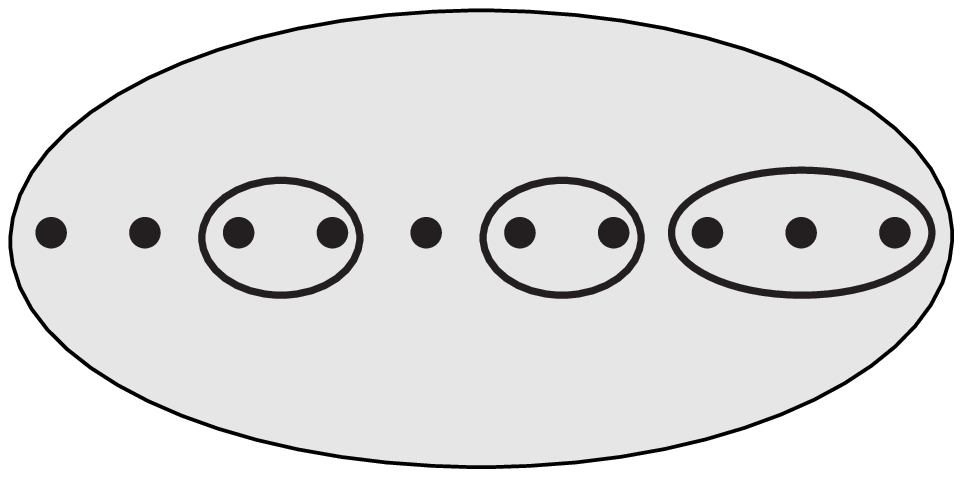}\\
(a) a standard curve system&&
(b) the unnested standard curve system $\C_{(1,1,2,1,2,3)}$
\end{tabular}
\vskip -2mm
\caption{Some standard curve systems in $D_{10}$}\label{fig:standard}
\end{figure}

$\Rext(\alpha)$ is unnested for any non-periodic reducible braid $\alpha$.
If $\alpha$ and $\beta$ are conjugate,
$\Rext(\alpha)$ and $\Rext(\beta)$ are of the same type.

\medskip

Recall that an ordered $k$-tuple $\n=(n_1,\ldots,n_k)$
is a $k$-composition of $n$ if $n=n_1+\cdots+n_k$ and
$n_i\ge 1$ for each $i$.
Unnested standard curve systems in $D_n$ are
in one-to-one correspondence with $k$-compositions of $n$
for $2\le k\le n-1$.
The $k$-braid group $B_k$ acts on the set of $k$-compositions as
$$
\alpha*(n_1,\ldots,n_k)
=(n_{\theta^{-1}(1)},\ldots,n_{\theta^{-1}(k)}),
$$
where $\alpha\in B_k$ and $\theta=\pi_\alpha$.

\begin{definition}\label{def:CurSys}
For a composition $\n=(n_1,\ldots,n_k)$ of $n$,
let $\C_\n$ denote the curve system $\cup_{n_i\ge 2}C_i$
in $D_n$, where $C_i$ with $n_i\ge 2$ is
a round circle enclosing the punctures
$\{m: \sum_{j=1}^{i-1}n_j< m\le \sum_{j=1}^{i}n_j\}$.
\end{definition}

For example, Figure~\ref{fig:standard}~(b)
shows $\C_\n$ for  $\n=(1,1,2,1,2,3)$.
For compositions $\n_1$ and $\n_2$ of $n$,
$\C_{\n_1}$ and $\C_{\n_2}$ are of the same type if and only if
$\n_1$ and $\n_2$ induce the same partition of $n$
such as $\n_1=(2,2,1)$ and $\n_2=(2,1,2)$.

\begin{definition}
Let $\n=(n_1,\cdots,n_k)$ be a composition of $n$.
\begin{itemize}
\item[(i)]
Let $\hat\alpha=l_1\cup\cdots\cup l_k$ be a $k$-braid,
where the right endpoint of $l_i$
is at the $i$-th position from bottom.
We define $\myangle{\hat\alpha}_\n$ as the $n$-braid obtained
from $\hat\alpha$ by taking $n_i$ parallel copies of $l_i$ for each $i$.
See Figure~\ref{fig:copy}~(a).

\item[(ii)]
Let $\alpha_i\in B_{n_i}$ for $i=1,\ldots,k$.
We define $(\alpha_1\oplus\cdots\oplus\alpha_k)_\n$ as the $n$-braid
$\alpha_1'\alpha_2'\cdots\alpha_k'$, where each $\alpha_i'$
is the image of $\alpha_i$ under the homomorphism
$B_{n_i}\to B_n$ defined by $\sigma_j\mapsto\sigma_{(n_1+\cdots+n_{i-1})+j}$.
See Figure~\ref{fig:copy}~(b).
\end{itemize}
\end{definition}

We will use the notation
$\alpha=\myangle{\hat\alpha}_\n(\alpha_1\oplus\cdots\oplus\alpha_k)_\n$
frequently.
See Figure~\ref{fig:copy}~(c).
Let $\bigoplus_{i=1}^m\alpha_i$ denote $\alpha_1\oplus\cdots\oplus\alpha_m$,
and let $m\alpha$ denote $\bigoplus_{i=1}^m\alpha$.
For example,
$$\arraycolsep=1pt
\begin{array}{rcl}
(\alpha_0\oplus {\textstyle\bigoplus_{i=1}^d}
(\alpha_{i,1}\oplus\cdots\oplus\alpha_{i,m}) )_\n
&=& (\alpha_0\oplus
(\alpha_{1,1}\oplus\cdots\oplus\alpha_{1,m})\oplus\cdots\oplus
(\alpha_{d,1}\oplus\cdots\oplus\alpha_{d,m}) )_\n,\\
(\alpha_0\oplus m\alpha_1\oplus
\cdots\oplus m\alpha_r)_\n
&=&(\alpha_0\oplus
(\underbrace{\alpha_1\oplus\cdots\oplus\alpha_1}_m )\oplus\cdots\oplus
(\underbrace{\alpha_r\oplus\cdots\oplus\alpha_r}_m ))_\n.
\end{array}
$$

\begin{figure}
\begin{tabular}{ccccc}
\includegraphics[scale=.8]{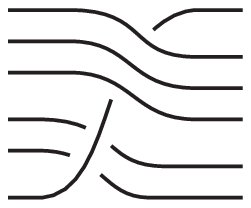} &\qquad&
\includegraphics[scale=.8]{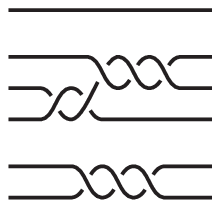} &\qquad&
\includegraphics[scale=.8]{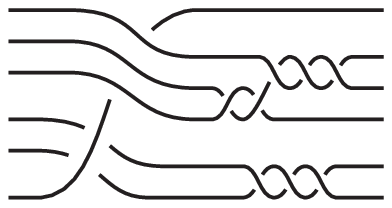} \\
\small (a) $\myangle{\sigma_1^{-1}\sigma_2}_\n$&&
\small (b) $(\sigma_1^3\oplus\sigma_1^{-2}\sigma_2^3\oplus 1)_\n$&&
\small (c) $\myangle{\sigma_1^{-1}
\sigma_2}_\n(\sigma_1^3\oplus\sigma_1^{-2}\sigma_2^3\oplus 1)_\n$
\end{tabular}
\caption{$\n=(2,3,1)$}\label{fig:copy}
\end{figure}

\begin{lemma}[{\cite[Lemmas 3.5 and 3.6]{LL08}}]\label{lem:decom}
Let\/ $\n=(n_1,\ldots,n_k)$ be a composition of\/ $n$.
\begin{enumerate}
\item[(i)]
The expression $\alpha=\myangle{\hat\alpha}_\n(\alpha_1\oplus\cdots\oplus\alpha_k)_\n$
is unique, that is, if\/
$\myangle{\hat\alpha}_\n(\alpha_1\oplus\cdots\oplus\alpha_k)_\n
=\myangle{\hat\beta}_\n(\beta_1\oplus\cdots\oplus\beta_k)_\n$,
then $\hat\alpha =\hat\beta$ and $\alpha_i=\beta_i$
for $1\le i\le k$.

\item[(ii)]
For $\alpha\in B_n$, $\alpha*\C_\n$ is standard if and only if\/
$\alpha$ can be expressed as
$\alpha=\myangle{\hat\alpha}_\n(\alpha_1\oplus\cdots\oplus\alpha_k)_\n$.
In this case, $\alpha*\C_\n=\C_{\hat\alpha*\n}$.

\item[(iii)]
$\myangle{\hat\alpha}_\n(\alpha_1\oplus\cdots\oplus\alpha_k)_\n
= (\alpha_{\theta^{-1}(1)}\oplus\cdots\oplus\alpha_{\theta^{-1}(k)})_{\hat\alpha\ast\n}
\myangle{\hat\alpha}_\n$,
where $\theta=\pi_{\hat\alpha}$.

\item[(iv)]
$\myangle{\hat\alpha \hat\beta}_\n
=\myangle{\hat\alpha}_{\hat\beta*\n}\myangle{\hat\beta}_\n$.

\item[(v)]
$(\myangle{\hat\alpha}_\n)^{-1}=\myangle{\hat\alpha^{-1}}_{\hat\alpha*\n}$.

\item[(vi)]
$
(\alpha_1\beta_1\oplus\cdots\oplus\alpha_k\beta_k)_\n
=(\alpha_1\oplus\cdots\oplus\alpha_k)_\n
(\beta_1\oplus\cdots\oplus\beta_k)_\n
$

\item[(vii)]
$(\alpha_1\oplus\cdots\oplus\alpha_k)_\n^{-1}
=(\alpha_1^{-1}\oplus\cdots\oplus\alpha_k^{-1})_\n$.

\item[(viii)]
$\myangle{\hat\alpha}_\n(\alpha_1\oplus\cdots\oplus\alpha_k)_\n
= \Delta_{(n)}$ if and only if\/
$\hat\alpha=\Delta_{(k)}$ and $\alpha_i=\Delta_{(n_i)}$ for $1\le i\le k$.
\end{enumerate}
\end{lemma}

Lemma~\ref{lem:decom}~(iii)--(vii)
are similar to the computations in the wreath product $G\wr B_k$,
where $k$-braids act on the $k$-cartesian product of $G$ by permuting coordinates.
Let $n=mk$ with $m,k\ge 2$, and let $\n=(m,\ldots,m)$ be a composition of $n$.
Let $B_n(\C_\n)=\{\alpha\in B_n\mid \alpha*\C_\n=\C_\n\}$.
Then Lemma~\ref{lem:decom} shows that
$B_n(\C_\n)$ is indeed isomorphic to $B_m\wr B_k$.

\begin{definition}
Let $\n=(n_1,\ldots,n_k)$ be a composition of $n$,
and let $\alpha\in B_n$ be such that $\alpha*\C_\n$ is standard.
Then $\alpha$ is uniquely expressed as
$\alpha=\myangle{\hat\alpha}_\n(\alpha_1\oplus\cdots\oplus\alpha_k)_\n$
by Lemma~\ref{lem:decom}.
We call $\hat\alpha$ the \emph{$\n$-exterior braid} of $\alpha$,
and denote it by $\Ext_\n(\alpha)$.
The braids $\alpha_1,\ldots,\alpha_k$ are called
\emph{$\n$-interior braids} of $\alpha$.
If $\Ext_\n(\alpha)$ is the identity,
$\alpha$ is called an \emph{$\n$-split braid}.
\end{definition}

The lemma below follows from the above lemma and properties of $\Rext(\cdot)$.

\begin{lemma}\label{lem:Rext}
Let $\alpha$ and $\beta$ be $n$-braids such that
$\alpha*\C_\n=\beta*\C_\n=\C_\n$
for a composition $\n$ of $n$.
\begin{itemize}
\item[(i)]
$\Ext_{\n}(\alpha\beta)=\Ext_{\n}(\alpha)\Ext_{\n}(\beta)$
and $\Ext_\n(\alpha^{-1})=\Ext_\n(\alpha)^{-1}$.
\item[(ii)]
If\/ $\Rext(\alpha)=\Rext(\beta)=\C_\n$ and $\gamma$ is an $n$-braid
with $\beta=\gamma\alpha\gamma^{-1}$, then
$\gamma*\C_\n=\C_\n$ and
$\Ext_\n(\beta)=\Ext_\n(\gamma)\Ext_\n(\alpha)\Ext_\n(\gamma)^{-1}$.
\end{itemize}
\end{lemma}

\section{Proof of Theorem~\ref{thm:main1}}

In this section, we prove Theorem~\ref{thm:main1} by presenting
three examples which show that
none of the monomorphisms
$\psi_1:A(\arB_n)\to A(\arA_n)$,
$\psi_2:A(\tilde \arA_{n-1}) \to A(\arB_n)$ and
$\psi_3:A(\tilde \arC_{n-1}) \to A(\arB_n)$
is injective on the set of conjugacy classes for $n\ge 3$.
Recall that
$$
\begin{array}{l}
A(\arA_n)=B_{n+1};\\
A(\arB_n)=B_{n+1,1};
\end{array}\quad
\begin{array}{l}
A(\tilde\arA_{n-1})=\{\alpha\in B_{n+1}\mid \mbox{$\alpha$ is 1-unlinked} \};\\
A(\tilde\arC_{n-1})=\{\alpha\in B_{n+1}\mid \mbox{$\alpha$ is $\{1,n+1\}$-pure} \}.
\end{array}
$$

\begin{proof}[Proof of Theorem~\ref{thm:main1}]
Consider the following $(n+1)$-braids for $n\ge 3$.
See Figure~\ref{fig:Nconj}.
$$
\biggl\{\begin{array}{l}
\alpha_1=\sigma_1^2\sigma_2^4,\\
\beta_1=\sigma_2^2\sigma_1^4,
\end{array}
\qquad
\biggl\{\begin{array}{l}
\alpha_2=\sigma_n^2\sigma_{n-1}^4,\\
\beta_2=\sigma_{n-1}^2\sigma_n^4,
\end{array}
\qquad\mbox{and}\qquad
\biggl\{\begin{array}{l}
\alpha_3=\sigma_n\cdots\sigma_2,\\
\beta_3=\sigma_1^{-2}\alpha_3\sigma_1^2.
\end{array}
$$

\begin{figure}\tabcolsep=4pt
\begin{tabular}{*7c}
\includegraphics[scale=.8]{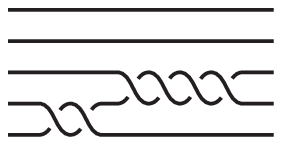} &&
\includegraphics[scale=.8]{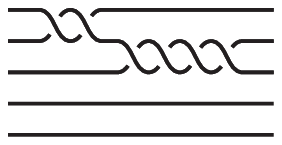} &&
\includegraphics[scale=.8]{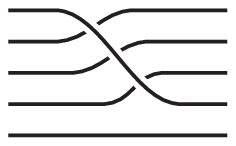} &&
\includegraphics[scale=.5]{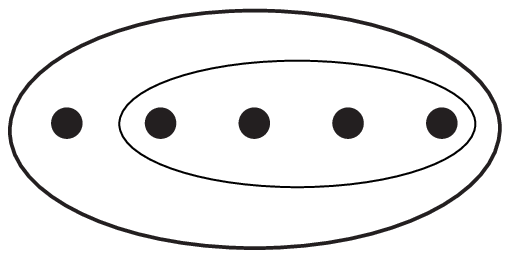}\\
(a) $\alpha_1=\sigma_1^2\sigma_2^4$ &&
(b) $\alpha_2=\sigma_n^2\sigma_{n-1}^4$ &&
(c) $\alpha_3=\sigma_n\cdots\sigma_2$ &&
(d) $\Rext(\alpha_3)$ \\[3mm]
\includegraphics[scale=.8]{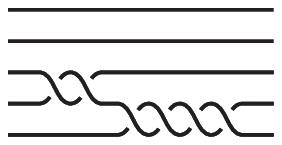} &&
\includegraphics[scale=.8]{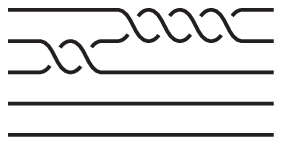} &&
\includegraphics[scale=.8]{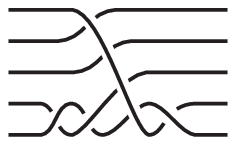} &&
\includegraphics[scale=.8]{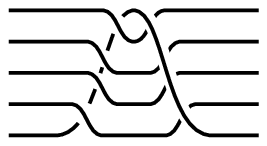}\\
(e) $\beta_1=\sigma_2^2\sigma_1^4$ &&
(f) $\beta_2=\sigma_{n-1}^2\sigma_n^4$ &&
(g) $\beta_3=\sigma_1^{-2}\alpha_3\sigma_1^2$ &&
(h) $\chi=(\sigma_1\cdots\sigma_n)(\sigma_n\cdots\sigma_1)$
\end{tabular}
\caption{The examples for Theorem~\ref{thm:main1} when $n=4$}\label{fig:Nconj}
\end{figure}

First, let us consider $\alpha_1$ and $\beta_1$.
Because $\alpha_1$ and $\beta_1$ are 1-pure,
they belong to $B_{n+1,1}$.
They are conjugate in $B_{n+1}$
because $\beta_1=(\sigma_1\sigma_2\sigma_1)\alpha_1
(\sigma_1\sigma_2\sigma_1)^{-1}$.
However they are not conjugate in $B_{n+1,1}$
because $\lk(\alpha_1)= 1$ and $\lk(\beta_1)= 2$.
This shows that $\psi_1:A(\arB_n)\to A(\arA_n)$
is not injective on the set of conjugacy classes.

\smallskip
The second pair $(\alpha_2,\beta_2)$ is similar to the first one.
Notice that $n\ge 3$.
Both $\alpha_2$ and $\beta_2$ are $\{1,n+1\}$-pure,
hence they belong to $A(\tilde\arC_{n-1})$.
They are conjugate in $B_{n+1,1}$ because
$\sigma_n\sigma_{n-1}\sigma_n$ is 1-pure and
$\beta_2=(\sigma_n\sigma_{n-1}\sigma_n)
\alpha_2(\sigma_n\sigma_{n-1}\sigma_n)^{-1}$.
However $\alpha_2$ cannot be conjugated to $\beta_2$ by
an $\{1,n+1\}$-pure braid.
An easy way to see this is to consider the linking number,
say $\lk'$, of $(n+1)$-st strand with the other strands,
which is a well-defined homomorphism
from $A(\tilde \arC_{n-1})$ to $\Z$.
Because $\lk'(\alpha_2)=1$ and $\lk'(\beta_2)=2$,
the braids $\alpha_2$ and $\beta_2$ are not conjugate in $A(\tilde \arC_{n-1})$.
This shows that $\psi_3:A(\tilde \arC_{n-1})\to A(\arB_n)$
is not injective on the set of conjugacy classes.

\smallskip
Now, let us consider the last pair $(\alpha_3,\beta_3)$.
Because both $\alpha_3$ and $\beta_3$ are 1-unlinked,
they belong to $A(\tilde \arA_{n-1})$.
They are conjugate in $B_{n+1,1}$
because $\sigma_1^2$ is 1-pure
and $\beta_3=\sigma_1^{-2}\alpha_3\sigma_1^2$.

The claim below shows the structure of $Z_{B_{n+1}}(\alpha_3)$.
It is a direct consequence of the main result of
J.~Gonz\'alez-Meneses and B.~Wiest in~\cite{GW04},
however we provide a proof for the reader's convenience.

\begin{claim}{Claim}
$Z_{B_{n+1}}(\alpha_3)$ is generated by $\alpha_3$ and
$\chi=(\sigma_1\cdots\sigma_n)(\sigma_n\cdots\sigma_1)$.
In particular, every element $\zeta$ in $Z_{B_{n+1}}(\alpha_3)$
is 1-pure with $\lk(\zeta)\equiv 0 \pmod n$.
\end{claim}

\begin{proof}[Proof of Claim]
Let $\n=(1,n)$ be a composition of $n+1$,
and let $Z(\alpha_3)=Z_{B_{n+1}}(\alpha_3)$.
Note that $\alpha_3=(1\oplus \delta_{(n)})_\n$ and $\Rext(\alpha_3)=\C_\n$.
See Figure~\ref{fig:Nconj}~(d).
If $\gamma\in Z(\alpha_3)$, then $\gamma*\C_\n=\C_\n$, hence
$\gamma$ is of the form
$$
\gamma=\myangle{\hat\gamma}_\n(1\oplus\gamma_1)_\n
$$
for some $\hat\gamma\in B_2$ and $\gamma_1\in B_n$.
By direct computation, we can see that
$\gamma\in Z(\alpha_3)$ if and only if
$\hat\gamma=\sigma_1^{2k}$ for some $k\in\Z$
and $\gamma_1$ commutes with $\delta_{(n)}$.
Because $Z_{B_n}(\delta_{(n)})$ is infinite cyclic generated by
$\delta_{(n)}$ by~\cite{BDM02},
$Z(\alpha_3)$ is generated by
$\myangle{\sigma_1^2}_\n=\chi$ and $(1\oplus\delta_{(n)})_\n=\alpha_3$.

Because $\lk(\chi)=n$ and $\lk(\alpha_3)=0$,
every element in $Z(\alpha_3)$
has linking number a multiple of $n$.
\end{proof}

Assume that $\alpha_3$ and $\beta_3$ are conjugate in $A(\tilde \arA_{n-1})$.
Then there exists a 1-unlinked $(n+1)$-braid $\gamma$ such that
$\alpha_3
=\gamma^{-1}\beta_3\gamma
=\gamma^{-1}\sigma_1^{-2}\alpha_3\sigma_1^2\gamma$.
Therefore $\sigma_1^2\gamma\in Z_{B_{n+1}}(\alpha_3)$.
Because $\lk(\sigma_1^2\gamma) =\lk(\sigma_1^2)+\lk(\gamma)=1+0=1$,
$\lk(\sigma_1^2\gamma)$ is not a multiple of $n\ge 3$.
It is a contradiction.
This shows that $\psi_2:A(\tilde \arA_{n-1})\to A(\arB_n)$
is not injective on the set of conjugacy classes.
\end{proof}

\section{Braids in the centralizer of a periodic braid}

In this section, we review isomorphisms between
Artin groups of type $\arB$ and centralizers
of periodic braids,
and make some tools useful in proving Theorem~\ref{thm:conj}.

\subsection{Discussions}

One of the difficulties in proving Theorem~\ref{thm:conj}
lies in dealing with reducible braids.
Let $(\alpha,\beta,\omega)$ be a triple of $n$-braids such that
$\omega$ is periodic and both $\alpha$ and $\beta$
are non-periodic reducible braids in $Z(\omega)$ which are conjugate in $B_n$.
If $\Rext(\alpha)$ and $\Rext(\beta)$
are standard, it would be very convenient for computations,
but they are not necessarily standard.
Our strategy is to construct another such triple $(\alpha',\beta',\omega')$
such that both $\Rext(\alpha')$ and $\Rext(\beta')$ are standard.
This new triple is related to the original one by conjugations
so that if $\alpha'$ and $\beta'$ are conjugate in $Z(\omega')$
then $\alpha$ and $\beta$ are conjugate in $Z(\omega)$.
This construction is the main theme of this section.

The following lemma is obvious, hence we omit the proof.

\begin{lemma}\label{lem:easy}
Let $\omega$ be a periodic $n$-braid,
and let $\alpha$ and $\beta$ belong to\/ $Z(\omega)$.
\begin{itemize}
\item[(i)]
Let $\chi\in B_n$.
Then $\chi\alpha\chi^{-1}$ and $\chi\beta\chi^{-1}$ belong to $Z(\chi\omega\chi^{-1})$.
Furthermore, $\chi\alpha\chi^{-1}$ and $\chi\beta\chi^{-1}$ are conjugate in
$Z(\chi\omega\chi^{-1})$ if and only if $\alpha$ and $\beta$ are conjugate in $Z(\omega)$.

\item[(ii)]
Let $\chi_1,\chi_2\in Z(\omega)$.
Then
$\chi_1\alpha\chi_1^{-1}$ and $\chi_2\beta\chi_2^{-1}$ belong to $Z(\omega)$.
Furthermore,
$\chi_1\alpha\chi_1^{-1}$ and $\chi_2\beta\chi_2^{-1}$ are conjugate in
$Z(\omega)$ if and only if $\alpha$ and $\beta$ are conjugate in $Z(\omega)$.
\end{itemize}
\end{lemma}

\smallskip
From the above lemma, one can think of two ways
in order to make $\Rext(\alpha)$ and $\Rext(\beta)$ standard.
First, one can expect that there exists $\chi\in B_n$
such that both $\chi*\Rext(\alpha)$ and $\chi*\Rext(\beta)$ are standard.
If it is true,
$(\chi\alpha\chi^{-1},\chi\beta\chi^{-1},\chi\omega\chi^{-1})$
can be taken as a desired triple by Lemma~\ref{lem:easy}~(i).
Second, one can expect that there exist $\chi_1,\chi_2\in Z(\omega)$
such that $\chi_1*\Rext(\alpha)$ and $\chi_2*\Rext(\beta)$ are standard.
If it is true,
$(\chi_1\alpha\chi_1^{-1},\chi_2\beta\chi_2^{-1},\omega)$
can be taken as a desired triple by Lemma~\ref{lem:easy}~(ii).
However the following example shows that
each of these ideas does not work for itself.
Our construction will be
as $(\chi\chi_1\alpha\chi_1^{-1}\chi^{-1}, \chi\chi_2\beta\chi_2^{-1}\chi^{-1},
\chi\omega\chi^{-1})$, where $\chi_1,\chi_2\in Z(\omega)$ and $\chi\in B_n$
such that both $\Rext(\chi\chi_1\alpha\chi_1^{-1}\chi^{-1})$ and
$\Rext(\chi\chi_2\beta\chi_2^{-1}\chi^{-1})$ are standard and
$\chi\omega\chi^{-1}$ is a special kind of periodic braid
(see Corollary~\ref{cor:sta}).

\begin{example}\label{ex:4st}
Consider the 4-braids
$$
\omega=\delta^2=(\sigma_3\sigma_2\sigma_1)^2,\qquad
\alpha=\sigma_2\sigma_1\sigma_2^{-1}
\quad\mbox{and}\quad
\beta=\sigma_3\sigma_2\sigma_3^{-1}.
$$
The element $\omega$ is a periodic 4-braid
which corresponds to the $\pi$-rotation of the punctured disk
when punctures are distributed as in Figure~\ref{fig:circ}~(a).
$\alpha$ and $\beta$ are reducible braids in $Z(\omega)$
which are conjugate in $B_4$.
Let $C=\Rext(\alpha)$ and $C'=\Rext(\beta)$.
They enclose two punctures as in Figure~\ref{fig:4st}~(a) and (b).

\begin{figure}\tabcolsep=10pt
\begin{tabular}{*5c}
\includegraphics[scale=.7]{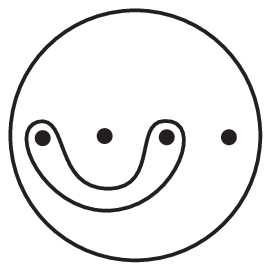} &
\includegraphics[scale=.7]{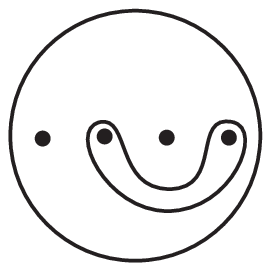} &
\includegraphics[scale=.7]{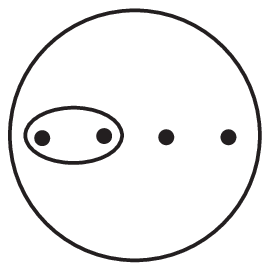} &
\includegraphics[scale=.7]{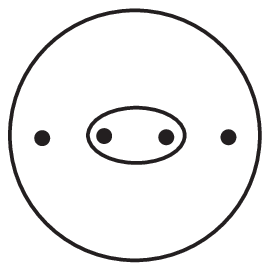} &
\includegraphics[scale=.7]{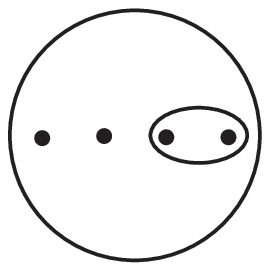} \\[1mm]
\includegraphics[scale=.7]{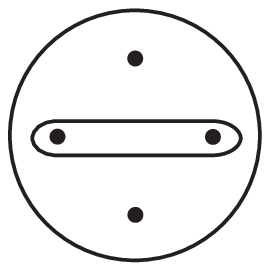} &
\includegraphics[scale=.7]{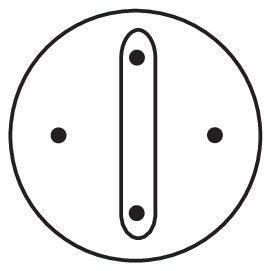} &
\includegraphics[scale=.7]{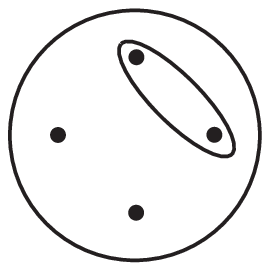} &
\includegraphics[scale=.7]{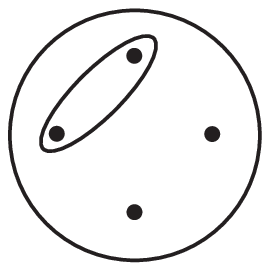} &
\includegraphics[scale=.7]{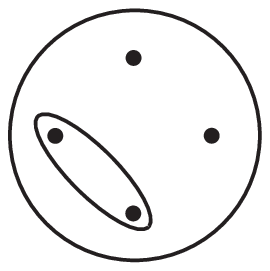}\\
(a) $C$ & (b) $C'$ & (c) & (d) & (e)
\end{tabular}
\caption{The figures in two rows show the same
curves with different location of punctures.}
\label{fig:4st}
\end{figure}

There are three standard curves in $D_4$
enclosing two punctures as in Figure~\ref{fig:4st}~(c)--(e)
and the geometric intersection numbers between two of them
are either 0 or 2.
Because the geometric intersection number between $C$ and $C'$ is 4,
there is no 4-braid $\chi$ such that
$\chi*C$ and $\chi*C'$ are standard.

Assume that there exists $\chi_1\in Z(\omega)$ such that $\chi_1*C$
is standard. Because $C$ encloses two punctures, so does $\chi_1*C$.
Therefore $\chi_1*C$ is one of the three curves
in Figure~\ref{fig:4st}~(c)--(e).
Moreover, $\chi_1*C$ is invariant under the action of $\omega$ because
$\omega*(\chi_1*C)
=\chi_1*(\omega*C)
=\chi_1*C$.
Because none of the curves in Figure~\ref{fig:4st}~(c)--(e)
is invariant under the action of $\omega$, it is a contradiction.
Similarly, there is no $\chi_2\in Z(\omega)$ such that $\chi_2*C'$ is standard.
\end{example}

The monomorphisms $\psi_5$ and $\psi_6$ are compositions of an isomorphism $\varphi$
and the inclusion $\iota$:
$$
\xymatrix{
B_{d+1,1} \ar[r]^\varphi_{\simeq} &
Z(\omega) \ar[r]^\iota &
B_n,
}
$$
where $n=md$ or $md+1$ with $m\ge 2$;
$\omega$ is a periodic $n$-braid conjugate to $\delta^k$ with $d=\gcd(k,n)$
or to $\epsilon^k$ with $d=\gcd(k,n-1)$.
Observe the following.
\begin{itemize}
\item[(i)]
$\iota$ is injective on the set of conjugacy classes if and only if
so is $\iota\circ\varphi$.
\item[(ii)]
Let $\omega'$ be a periodic $n$-braid which is conjugate to $\omega$,
and let $\iota':Z(\omega')\to B_n$ be the inclusion.
Then $\iota$ and $\iota'$ are related by an inner automorphism, say $I_\chi$,
of $B_n$ and its restriction to $Z(\omega)$ as follows.
$$
\xymatrix{
Z(\omega) \ar[d]_{\simeq}^{I_\chi} \ar[r]^{\iota} & B_n \ar[d]_{\simeq}^{I_\chi} \\
Z(\omega') \ar[r]^{\iota'} & B_n
}
$$
Therefore, $\iota$ is injective on the set of conjugacy classes
if and only if so is $\iota'$.
\item[(iii)]
Consequently, the injectivity of $\iota\circ\varphi$ on the set of conjugacy classes
depends only on the conjugacy class of $\omega$.
\end{itemize}

Therefore Theorem~\ref{thm:conj} states that
$\iota: Z(\omega)\to B_n$ is injective on the set of conjugacy classes
for any periodic $n$-braid $\omega$.
And Corollary~\ref{thm:main2} is equivalent to
$\iota: Z(\omega)\to B_n$ being injective on the set of conjugacy classes
for any noncentral periodic $n$-braid $\omega$.
Because Theorem~\ref{thm:conj} obviously holds for
central braids $\omega$, we have the following.

\begin{lemma}\label{lem:equiv}
Theorem~\ref{thm:conj} and Corollary~\ref{thm:main2} are equivalent.
\end{lemma}

We remark that there are infinitely many isomorphisms from $B_{d+1,1}$ to $Z(\omega)$,
and any two of them are related by an automorphism of
$B_{d+1,1}=A(\arB_d)$. See~\cite{CC05} for the classification
of automorphisms of Artin groups of type $\arA$, $\arB$,
$\tilde\arA$ and $\tilde\arC$.

\subsection{Periodic braids and their centralizers}
\label{sec:def-mu}

Let $m$ and $d$ be integers with $m\ge 2$ and $d\ge 1$, and let $n=md+1$.
Here we construct a periodic $n$-braid $\mu_{m,d}$
conjugate to $\epsilon^d$,
and an isomorphism $B_{d+1,1}\simeq Z(\mu_{m,d})$
following the work in~\cite{BDM02}.

Let $Z_n^{(m)}=\{z_0\}\cup \{z_{i,j}\mid 1\le i\le d,~1\le j\le m\}$,
where $z_0$ is the origin and $z_{i,j}$ is the complex number
with $|z_{i,j}|=i$ and $\arg(z_{i,j})=2\pi (j-1)/m$.
Let
$$
D_n^{(m)} =\{z\in\mathbb C\mid |z|\le n+1\}
\setminus Z_n^{(m)}.
$$
\begin{figure}
$\xymatrix{
\includegraphics[scale=.7]{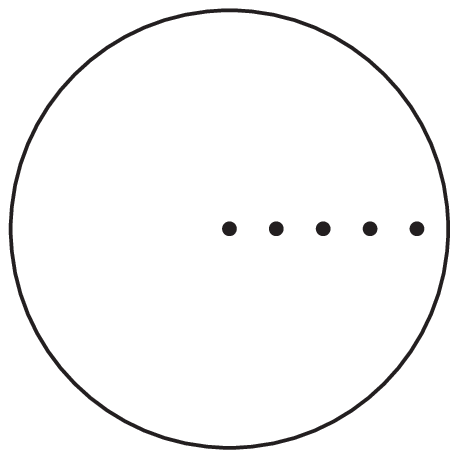} &
\includegraphics[scale=.7]{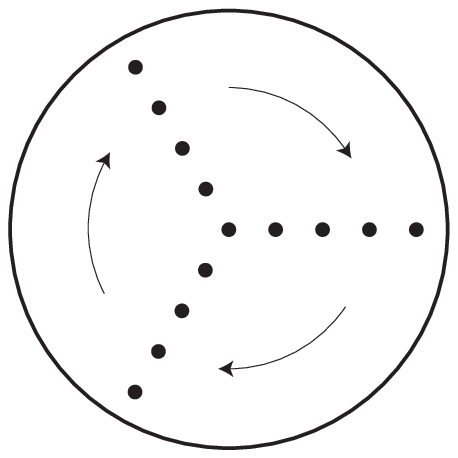}
    \ar[r]_{\eta_{m,d}\qquad \mbox{}}
    \ar[l]^{\phi_{m,d}\mbox{}}
&\includegraphics[scale=.7]{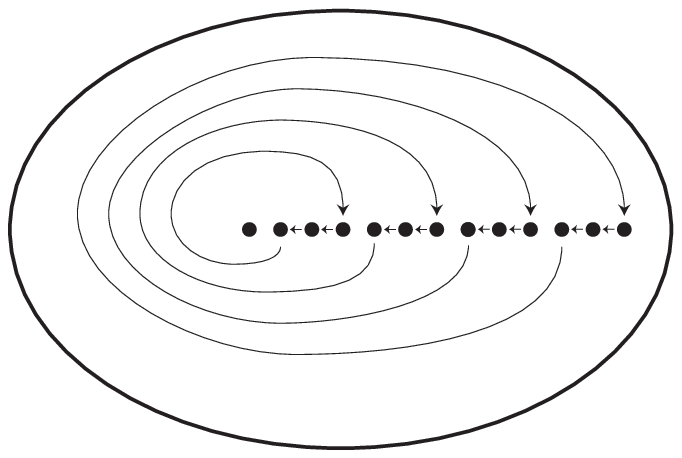}
}$\par
\mbox{}\hspace{10mm} $D_{d+1}$
\mbox{}\hspace{25mm} $D_n^{(m)}$ and $\rho_m$
\mbox{}\hspace{20mm} $D_{n}$ and $\eta_{m,d}\circ\rho_m\circ\eta_{m,d}^{-1}$
\caption{The punctured spaces $D_n^{(m)}$ and $D_{d+1}$,
where $n=13$, $m=3$ and $d=4$}
\label{fig:punc-ray}
\end{figure}
Now we construct a covering map $\phi_{m,d}: D_n^{(m)}\to D_{d+1}$
and a diffeomorphism $\eta_{m,d}: D_n^{(m)}\to D_n$ as in Figure~\ref{fig:punc-ray}.
Let $\rho_m$ be the clockwise $(2\pi/m)$-rotation of $D_n^{(m)}$.
The identification space $D_n^{(m)}/\langle\rho_m\rangle$
is diffeomorphic
to the $(d+1)$-punctured disk $D_{d+1}$.
We write the natural projection by
$$
\phi_{m,d}: D_n^{(m)}\to D_{d+1}.
$$
$\phi_{m,d}$ is a covering map
whose group of covering transformations is
the cyclic group generated by $\rho_m$.
We construct a diffeomorphism
$$
\eta_{m,d}: D_n^{(m)}\to D_n,
$$
which fixes $\partial D_n^{(m)}=\partial D_n$ pointwise, as follows.
First, cut $D_n^{(m)}$ and $D_n$ by $d$ round circles
into one once-punctured disk and $d$ annuli with $m$ punctures
as in Figure~\ref{fig:diffeo}~(a).
Then, define $\eta_{m,d}$ on each punctured annulus
as in Figure~\ref{fig:diffeo}~(b) such that it coincides
at the intersections of two adjacent annuli,
and extend $\eta_{m,d}$ to the once-punctured disk in an obvious way.
Such a diffeomorphism $\eta_{m,d}$ is unique up to isotopy.

\begin{figure}
\begin{tabular}{cc}
$\xymatrix{
\includegraphics[scale=.6]{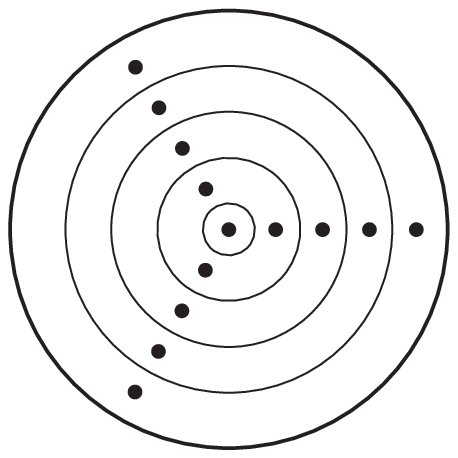}\ar[r]^{\eta_{m,d}}&
\includegraphics[scale=.6]{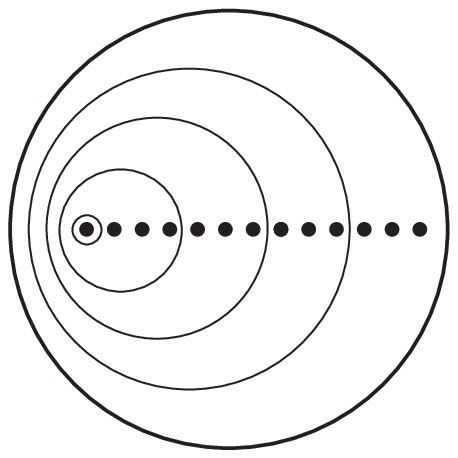}}$ &
$\xymatrix{
\includegraphics[scale=.6]{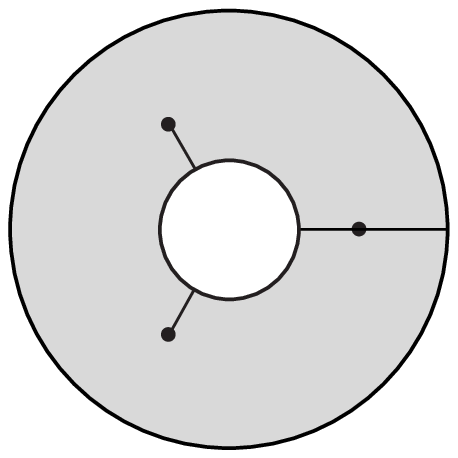}\ar[r]^{\eta_{m,d}}&
\includegraphics[scale=.6]{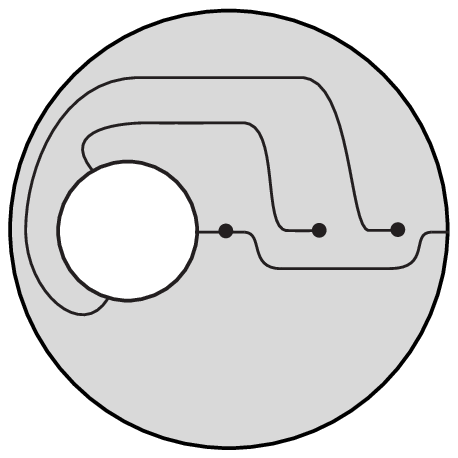}}$\\
(a) & (b)
\end{tabular}
\caption{The diffeomorphism $\eta_{m,d}: D_{n}^{(m)} \to D_{n}$}
\label{fig:diffeo}
\end{figure}

\begin{definition}
Let $m$ and $d$ be integers with $m\ge 2$ and $d\ge 0$.
We define an $(md+1)$-braid $\mu_{m,d}$ recursively as follows:
$\mu_{m,0}$ is the unique braid with one strand;
for $d\ge 1$,
$$\mu_{m,d}=\mu_{m,d-1}(\sigma_{dm}\sigma_{dm-1}\cdots\sigma_2\sigma_1)
(\sigma_1\sigma_2\cdots\sigma_{(d-1)m+1}),
$$
where $\mu_{m,d-1}$ is regarded as an $(md+1)$-braid
by using the inclusion $B_{m(d-1)+1}\to B_{md+1}$
defined by $\sigma_i\mapsto \sigma_i$ for $1\le i\le m(d-1)$.
See Figure~\ref{fig:mu} for $\mu_{m,d}$ with $m=3$ and $1\le d\le 4$.
\end{definition}

\begin{figure}
\tabcolsep=.7em
\begin{tabular}{cccc}
\includegraphics[scale=.6]{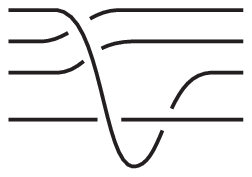} &
\includegraphics[scale=.6]{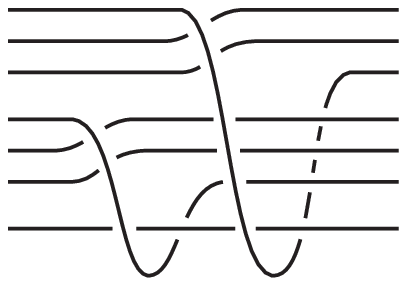} &
\includegraphics[scale=.6]{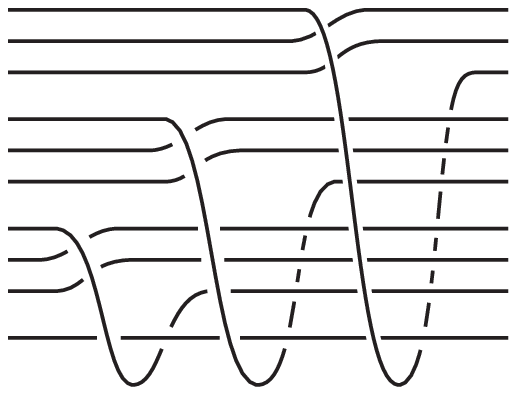} &
\includegraphics[scale=.6]{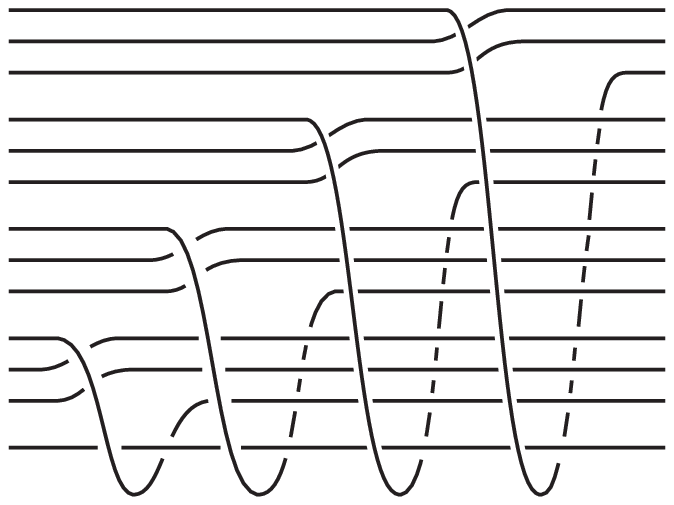} \\
$\mu_{3,1}$ &
$\mu_{3,2}$ &
$\mu_{3,3}$ &
$\mu_{3,4}$
\end{tabular}
\caption{The periodic braid $\mu_{3,d}$ for $d=1,2,3,4$}
\label{fig:mu}
\end{figure}

From the construction, we can see the following.
\begin{itemize}
\item[(i)]
$\eta_{m,d}\circ\rho_m\circ\eta_{m,d}^{-1}: D_n\to D_n$ represents
the braid $\mu_{m,d}$ modulo $\Delta^2$.

\item[(ii)]
An $n$-braid belongs to $Z(\mu_{m,d})$
if and only if it is represented by $\eta_{m,d}\circ h\circ \eta_{m,d}^{-1}$
for a $\rho_m$-equivariant diffeomorphism $h:D_n^{(m)}\to D_n^{(m)}$
that fixes the boundary pointwise.
\end{itemize}

\smallskip
Because $(\mu_{m,d})^m=\Delta_{(n)}^2=(\epsilon_{(n)}^d)^m$,
$\mu_{m,d}$ is conjugate to $\epsilon_{(n)}^d$, hence
$\nu(\mu_{m,d})$ is conjugate to $\delta_{(n-1)}^d$.
Recall that the monomorphisms $\psi_4$, $\psi_5$ and $\psi_6$ are
induced by the isomorphisms $B_{d+1,1}\simeq Z(\epsilon_{(n)}^d)$
and $B_{d+1,1}\simeq Z(\delta_{(n-1)}^d)$ in~\cite{BDM02,GW04}.
In the same way, we define monomorphisms $\psi_4'$, $\psi_5'$ and $\psi_6'$
by using isomorphisms $B_{d+1,1}\simeq Z(\mu_{m,d})$
and $B_{d+1,1}\simeq Z(\nu(\mu_{m,d}))$.

\begin{definition}
Let $m\ge 2$ and $d\ge 1$, and let $n=md+1$.
We define an isomorphism
$$
\varphi:B_{d+1,1}\to Z(\mu_{m,d})
$$
as follows.
Let $\alpha\in B_{d+1,1}$, that is,
$\alpha$ is a 1-pure $(d+1)$-braid.
Then $\alpha$ is represented by a diffeomorphism
$f:D_{d+1}\to D_{d+1}$ that fixes the first puncture
together with all the boundary points.
Let $\tilde f:D_n^{(m)}\to D_n^{(m)}$ be the lift of $f$ (with respect to $\phi_{m,d}$)
that fixes $\partial D_n^{(m)}$ pointwise.
Then $\tilde f$ is $\rho_m$-equivariant, hence
$\eta_{m,d}\circ \tilde f\circ\eta_{m,d}^{-1}$ represents
a braid in $Z(\mu_{m,d})$.
We define $\varphi(\alpha)$ as the braid represented by
$\eta_{m,d}\circ \tilde f\circ\eta_{m,d}^{-1}$.
$$
\xymatrix{
D_n \ar[r]^{\eta_{m,d}^{-1} } &
D_n^{(m)} \ar[d]^{\phi_{m,d}} \ar[r]^{\tilde f} &
D_n^{(m)} \ar[d]^{\phi_{m,d}} \ar[r]^{\eta_{m,d}} & D_n\\
& D_{d+1}\ar[r]^{f} & D_{d+1}
}
$$
The monomorphisms $\psi_4'$ and $\psi_6'$ are compositions
of $\varphi$ with the inclusions $Z(\mu_{m,d})\subset B_{md+1,1}$
and $Z(\mu_{m,d})\subset B_{md+1}$, respectively.
The monomorphism $\psi_5'$ is the composition of $\varphi$ with
the homomorphism $\nu:B_{md+1,1}\to B_{md}$.
See Figure~\ref{fig:inj2}.
\end{definition}

\begin{figure}
$
\xymatrix{
B_{md} &
Z(\nu(\mu_{m,d})) \ar@{_(->}[l] &
Z(\mu_{m,d}) \ar[l]_{\mbox{}\quad\nu} \ar@{^(->}[r] &
B_{md+1,1} \ar@{^(->}[r] &
B_{md+1} \\
& & B_{d+1,1}
\ar[u]_{\varphi}^{\simeq}
\ar[llu]^{\psi_5'}
\ar[lu]_{\simeq}
\ar[ru]^{\psi_4'}
\ar[rru]_{\psi_6'}
}
$
\caption{Monomorphisms $\psi_4'$, $\psi_5'$ and $\psi_6'$}
\label{fig:inj2}
\end{figure}
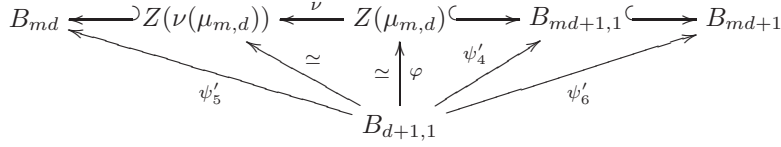

By the construction, $\psi_4'(\Delta_{(d+1)}^2)=\psi_6'(\Delta_{(d+1)}^2)=\mu_{m,d}$.
The following lemma looks obvious,
but we include a sketchy proof for completeness.

\begin{lemma}\label{lem:dyn}
Let $\psi$ denote either $\psi_4':B_{d+1,1}\to B_{md+1,1}$,
$\psi_5':B_{d+1,1}\to B_{md}$ or $\psi_6':B_{d+1,1}\to B_{md+1}$.
The monomorphism $\psi$ preserves the Nielsen-Thurston type,
that is, for $\alpha\in B_{d+1,1}$
\begin{itemize}
\item[(i)]
$\alpha$ is periodic if and only if $\psi(\alpha)$ is periodic;
\item[(ii)]
$\alpha$ is pseudo-Anosov if and only if $\psi(\alpha)$ is pseudo-Anosov;
\item[(iii)]
$\alpha$ is reducible and non-periodic
if and only if $\psi(\alpha)$ is reducible and non-periodic.
\end{itemize}
\end{lemma}

\begin{proof}
Assume $\psi=\psi_6'$.
The case of $\psi_4'$ is the same, and the case of
$\psi_5'$ can be proved similarly.
Because the three classes \{periodic braids\},
\{pseudo-Anosov braids\} and
\{non-periodic reducible braids\}
are mutually disjoint,
it suffices to prove (i) together with (ii)$'$ and (iii)$'$ where
\begin{itemize}
\item[(ii)$'$]
If $\alpha$ is pseudo-Anosov, then $\psi(\alpha)$ is pseudo-Anosov;
\item[(iii)$'$]
If $\alpha$ is reducible and non-periodic,
then $\psi(\alpha)$ is reducible and non-periodic.
\end{itemize}

\smallskip(i)\ \
\begin{tabular}[t]{lll}
$\alpha$ is periodic
&$\Leftrightarrow$&
$\alpha^k=(\Delta^2)^{\ell}$
for some $k\neq 0$ and $\ell\in\Z$\\
&$\Leftrightarrow$&
$\psi(\alpha)^k=\psi(\alpha^k)=\psi(\Delta^{2\ell})
=\mu_{m,d}^\ell$
for some $k\neq 0$ and $\ell\in\Z$\\
&$\Leftrightarrow$&
$\psi(\alpha)$ is periodic
\end{tabular}

\smallskip(ii)$'$\ \
Suppose that $\alpha\in B_{d+1,1}$ is pseudo-Anosov.
There is a pseudo-Anosov diffeomorphism, say $f$,
defined on the interior of $D_{d+1}$
representing $\alpha$ modulo $\Delta_{(d+1)}^{2}$.
Let $\tilde f$ be the lift of $f$
(with respect to $\phi_{m,d}$) defined on the interior of $D_{md+1}^{(m)}$.
Then $\tilde f$ is also a pseudo-Anosov diffeomorphism because
invariant measured foliations of $f$ lift to
invariant measured foliations of $\tilde f$
with the same dilatation.
Therefore $\eta_{m,d}\circ \tilde f\circ \eta_{m,d}^{-1}$
is a pseudo-Anosov diffeomorphism.
Since $\eta_{m,d}\circ \tilde f\circ \eta_{m,d}^{-1}$ represents $\psi(\alpha)$
modulo $\Delta_{(md+1)}^{2}$,
$\psi(\alpha)$ is pseudo-Anosov.

\smallskip(iii)$'$\ \
Suppose that $\alpha\in B_{d+1, 1}$ is reducible and non-periodic.
Then $\psi(\alpha)$ is non-periodic by (i).
Because $\alpha$ is reducible, there exists
a diffeomorphism $f:D_{d+1}\to D_{d+1}$, which represents $\alpha$, together with
an essential curve system $\C$ such that $f(\C)=\C$.
Let $\tilde \C=\phi_{m,d}^{-1}(\C)$,
and let $\tilde f$ be a lift of $f$ with respect to $\phi_{m,d}$.
Then $\tilde \C$ is an invariant essential curve system of $\tilde f$,
hence $\eta_{m,d}(\tilde \C)$ is a reduction system of $\psi(\alpha)$.
\end{proof}

\begin{corollary}\label{cor:dyn}
Let $m\ge 2$ and $d\ge 1$.
Let $\omega$ be a 1-pure braid conjugate to $\epsilon_{(md+1)}^d$,
and let $\alpha\in Z(\omega)$.
Then $\alpha$ and $\nu(\alpha)$ have the same Nielsen-Thurston type,
that is,
\begin{itemize}
\item[(i)]
$\alpha$ is periodic if and only if $\nu(\alpha)$ is periodic;
\item[(ii)]
$\alpha$ is pseudo-Anosov if and only if $\nu(\alpha)$ is pseudo-Anosov;
\item[(iii)]
$\alpha$ is reducible and non-periodic
if and only if $\nu(\alpha)$ is reducible and non-periodic.
\end{itemize}
\end{corollary}

\begin{proof}
Without loss of generality, we may assume that $\omega=\mu_{m,d}$.
Then there exists $\bar\alpha\in B_{d+1,1}$ such that
$\alpha=\psi_4'(\bar\alpha)$ and $\nu(\alpha)=\psi_5'(\bar\alpha)$.
By Lemma~\ref{lem:dyn}, $\alpha$ and $\bar\alpha$
(resp.\ $\nu(\alpha)$ and $\bar\alpha$) have the same Nielsen-Thurston type.
Therefore $\alpha$ and $\nu(\alpha)$ have the same Nielsen-Thurston type.
\end{proof}

From Lemma~\ref{lem:dyn} and its proof, we can see that
the monomorphisms $\psi_4'$, $\psi_5'$ and $\psi_6'$ send
canonical reduction systems to canonical reduction systems
as follows.

\begin{corollary}\label{cor:crs}
Let $m\ge 2$ and $d\ge 1$,
and let $\alpha\in B_{d+1,1}$ be reducible and non-periodic.
If\/ $\C$ is the canonical reduction system of\/ $\alpha$,
then $\eta_{m,d}(\phi_{m,d}^{-1}(\C))$ is the canonical reduction system
of both $\psi_4'(\alpha)$ and $\psi_6'(\alpha)$,
and the canonical reduction system of $\psi_5'(\alpha)$ is obtained from
$\eta_{m,d}(\phi_{m,d}^{-1}(\C))$ by forgetting the first puncture.
In particular, $\Rext(\psi_4'(\alpha))=\Rext(\psi_6'(\alpha))=
\eta_{m,d}(\phi_{m,d}^{-1}(\Rext(\alpha)))$,
and $\Rext(\psi_5'(\alpha))$ is obtained from $\Rext(\psi_4'(\alpha))$ by
forgetting the first puncture.
\end{corollary}

\begin{proof}\def\D{\mathcal D}
We prove the corollary only for the case of $\psi_4'$.
The other cases are similar.
Recall that the canonical reduction system of a braid is
the smallest adequate reduction system of it.

Let $\C\subset D_{d+1}$ be the canonical reduction system of $\alpha\in B_{d+1,1}$.
Following the construction in Lemma~\ref{lem:dyn},
let $\tilde \C=\phi_{m,d}^{-1}(\C)$ and $\D=\eta_{m,d}(\tilde \C)$.

From the proof of Lemma~\ref{lem:dyn}, $\D$ is a reduction system
of $\psi_4'(\alpha)$.
Because $\C$ is adequate, $\D$ is adequate by Lemma~\ref{lem:dyn}~(i) and (ii).
Hence there is a subset $\D'$ of $\D$
which is the canonical reduction system of $\psi_4'(\alpha)$.
Let $\tilde \C'=\eta_{m,d}^{-1}(\D')$ and $\C'=\phi_{m,d}(\tilde \C')$.
Note that $\tilde\C'\subset\tilde \C$ and $\C'\subset \C$.
$$
\xymatrix{
D_n^{(m)} \ar[d]^{\phi_{m,d}} \ar[r]^{\eta_{m,d}} & D_n\\
D_{d+1}
}
\qquad
\xymatrix{
\tilde \C'\subset \tilde \C
\ar[d]^{\phi_{m,d}}
\ar[r]^{\eta_{m,d}} &
\D'\subset \D\\
\C'\subset \C
}
$$

We claim that $\tilde \C'=\phi_{m,d}^{-1}(\C')$ and that $\C'$ is a reduction
system of $\alpha$.

Since $\D'$ is the canonical reduction system of $\psi_4'(\alpha)$
and since $\psi_4'(\alpha)$ commutes with $\mu_{m,d}$,
$$
\mu_{m,d}*\D'
=\mu_{m,d}*\R(\psi_4'(\alpha))
=\R(\mu_{m,d}\,\psi_4'(\alpha)\,\mu_{m,d}^{-1})
=\R(\psi_4'(\alpha))
=\D'.
$$
Because the diffeomorphism $\eta_{m,d}\circ\rho_m\circ \eta_{m,d}^{-1}$
represents $\mu_{m,d}\bmod \Delta^2$,
the above formula implies that $\rho_m(\tilde \C')=\tilde \C'$,
that is, if $\tilde\C'$ contains a component of $\tilde\C$ then it contains
all the components of $\tilde\C$ in the $\rho_m$-orbit of this component.
Since the covering transformation group of $\phi_{m,d}$ is generated by $\rho_m$,
$$
\tilde \C'=\phi_{m,d}^{-1}(\C')
\quad\mbox{and hence}\quad
\D'=\eta_{m,d}(\phi_{m,d}^{-1}(\C')).
$$

Let $f:D_{d+1}\to D_{d+1}$ be a diffeomorphism representing $\alpha$
such that $f(\C)=\C$.
Let $\tilde f$ be a lift of $f$ with respect to $\phi_{m,d}$.
Then $\tilde f(\tilde\C)=\tilde\C$ because $\tilde \C=\phi_{m,d}^{-1}(\C)$.
The diffeomorphism $\eta_{m,d}\circ\tilde f\circ \eta_{m,d}^{-1}$ represents
the braid $\psi_4'(\alpha)$.
Since $\D'=\eta_{m,d}(\tilde \C')$ is a reduction system of $\psi_4'(\alpha)$,
one has $\tilde f(\tilde\C')=\tilde\C'$.
Therefore $f(\C')=\C'$ because $\tilde \C'=\phi_{m,d}^{-1}(\C')$ and
$\tilde f$ is a lift of $f$ with respect to $\phi_{m,d}$.
Hence $\C'$ is a reduction system of $\alpha$.

\medskip

Since $\D'$ is adequate,
the reduction system $\C'$ is also adequate by Lemma~\ref{lem:dyn}~(i) and (ii).
Because $\C$ is the canonical reduction system of $\alpha$,
it is the smallest adequate reduction system of $\alpha$, hence $\C\subset \C'$.
Because $\C'\subset \C$, this implies $\C=\C'$ and hence $\D=\D'$.

So far, we have shown that
if $\C$ is the canonical reduction system of $\alpha$
then $\eta_{m,d}(\phi_{m,d}^{-1}(\C))$
is the canonical reduction system of $\psi_4'(\alpha)$.
If $\C''$ is the collection of the outermost components of $\C$,
then $\eta_{m,d}(\phi_{m,d}^{-1}(\C''))$ is the collection of the outermost
components of $\eta_{m,d}(\phi_{m,d}^{-1}(\C))$,
hence $\Rext(\psi_4'(\alpha))=\eta_{m,d}(\phi_{m,d}^{-1}(\Rext(\alpha)))$.
\end{proof}

We remark that the above argument actually shows that
there is a one-to-one correspondence
$$
\{\mbox{curve systems $\C\subset D_{d+1}$}\}
\stackrel{1-1}{\longleftrightarrow}
\{\mbox{curve systems $\D\subset D_{md+1}$ such that $\mu_{m,d}*\D=\D$}\}
$$
sending $\C\subset D_{d+1}$ to $\eta_{m,d}(\phi_{m,d}^{-1}(\C))$.
This correspondence satisfies the property that
$\C$ is a reduction system (resp.~an adequate reduction system,
the canonical reduction system) of $\alpha\in B_{d+1}$
if and only if
$\eta_{m,d}(\phi_{m,d}^{-1}(\C))$ is a reduction system
(resp.~an adequate reduction system,
the canonical reduction system) of $\psi_4'(\alpha)$.

\medskip

For $1\le i\le d$ and $1\le j\le m$,
let $x_0$ and $x_{i,j}$ be the integers such that
$x_0=1$ and $x_{i,j}=(i-1)m+j+1$.
In other words,
$(1,2,\ldots,md+1)
=(x_0,x_{1,1},\ldots,x_{1,m},
\ldots,x_{d,1},\ldots,x_{d,m})$.
Note that $\eta_{m,d}(z_0)=x_0$ and
$\eta_{m,d}(z_{i,j})=x_{i,j}$ for $1\le i\le d$ and $1\le j\le m$.
The induced permutation of $\mu_{m,d}$ is
$$
\pi_{\mu_{m,d}}=(x_0)(x_{1,m},\ldots, x_{1,2},x_{1,1})
\cdots(x_{d,m},\ldots, x_{d,2},x_{d,1}).
$$

\begin{lemma}
\label{lem:per}
Let $m\ge 2$ and $d\ge 1$, and let
$\theta$ be a permutation on $\{1,\ldots,md+1\}$.
Then the following conditions are equivalent.
\begin{itemize}
\item[(i)]
$\theta$ is an induced permutation of an element of $Z(\mu_{m,d})$.
\item[(ii)]
$\theta$ commutes with the induced permutation of $\mu_{m,d}$
\item[(iii)]
$\theta$ is such that
$\theta(x_0)=x_0$ and
$\theta(x_{i,j})=x_{k_i, j-\ell_i}$ for $1\le i\le d$ and $1\le j\le m$,
where $0\le \ell_i< m$ and $k_1,\ldots, k_d$ is a rearrangement
of $1,\ldots,d$.
(Hereafter we regard the second index $j-\ell_i$ of $(k_i, j-\ell_i)$
as being taken modulo $m$.)
\end{itemize}
\end{lemma}

\begin{proof}
(i) $\Rightarrow$ (ii) $\Leftrightarrow$ (iii) is clear.
Let us show that (ii) $\Rightarrow$ (i).
Suppose that (ii) is true.

The centralizer of $\pi_{\mu_{m,d}}$ is generated by two types of
permutations $\theta_i$ for $1\le i\le d-1$ and
$\tau_i$ for $1\le i\le d$, where
\begin{itemize}
\item[(i)]
the permutation $\theta_i$ interchanges
the $i$-th and the $(i+1)$-st cycles of $\pi_{\mu_{m,d}}$, more precisely,
$\theta_i(x_{i,k})=x_{i+1, k}$,
$\theta_i(x_{i+1, k})=x_{i,k}$ and
$\theta_i(x_{j,k})=x_{j,k}$ for $1\le k\le m$ and $j\not\in\{i,i+1\}$;
\item[(ii)]
the permutation
$\tau_i$ shifts the $i$-th cycle of $\pi_{\mu_{m,d}}$ by one
in a way that $\tau_i(x_{i,k})=x_{i, k-1}$ and
$\tau_i(x_{j,k})=x_{j,k}$ for $1\le k\le m$ and $j\ne i$.
\end{itemize}
Notice that $\theta_i$ and $\tau_i$
are induced permutations of elements of $Z(\mu_{m,d})$.
Let the punctures be located as in Figure~\ref{fig:per}~(a)
so that $\mu_{m,d}$ corresponds to the clockwise $2\pi/m$-rotation.
Figure~\ref{fig:per}~(b) and (c) illustrate
diffeomorphisms whose induced permutations
interchange two cycles or shift a cycle by one.
\end{proof}

\begin{figure}
\begin{tabular}{ccccc}
\includegraphics[scale=.65]{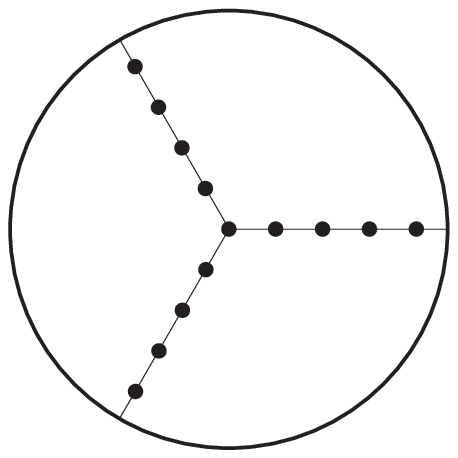} &\mbox{}\quad\mbox{}&
\includegraphics[scale=.65]{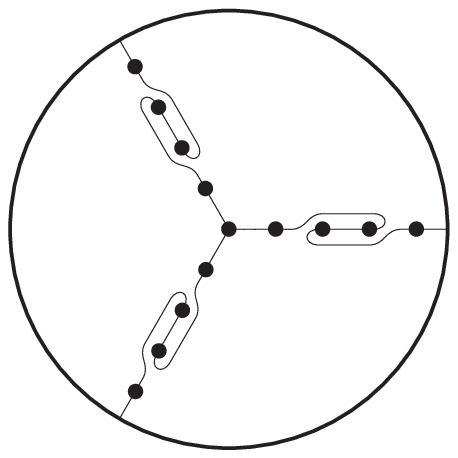} &\mbox{}\quad\mbox{}&
\includegraphics[scale=.65]{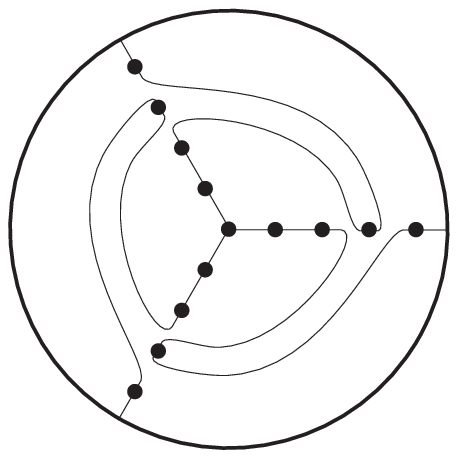} \\
(a)  &&
(b)  &&
(c)
\end{tabular}
\caption{(a) shows the punctured disk $D_{13}^{(3)}$.
(b) and (c) illustrate diffeomorphisms inducing $\theta_2$ and $\tau_3$, respectively.
}
\label{fig:per}
\end{figure}

\subsection{Reducible braids in the centralizer of a periodic braid}
\label{sec:red-sta}
Let $m$ and $d$ be integers with $m\ge 2$ and $d\ge 0$, and let $n=md+1$.
Let $\d=(d_0+1,d_1,\ldots,d_r)$ be a composition of $d+1$.
We define a composition $L_m(\mathbf d)$ of $n$ as
$$
L_m(\mathbf d)
=(md_0+1,\underbrace{d_1,\ldots,d_1}_{m},\ldots,
\underbrace{d_r,\ldots,d_r}_{m}).
$$
For example,
if $\mathbf d=(2,2,1)$ then $L_3(\mathbf d)=(4,2,2,2,1,1,1)$, and
if $\mathbf d=(3,2,1)$ then $L_3(\mathbf d)=(7,2,2,2,1,1,1)$.
Let $\C_{m,\d}$ denote the curve system $\phi_{m,d}^{-1}(\C_\d)$ in $D_n^{(m)}$.
See Figure~\ref{fig:cur-ray}.
\begin{figure}
$$
\xymatrix{
\includegraphics[scale=.6]{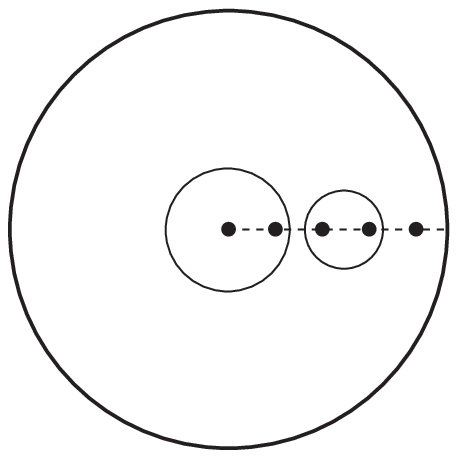}
&
\includegraphics[scale=.6]{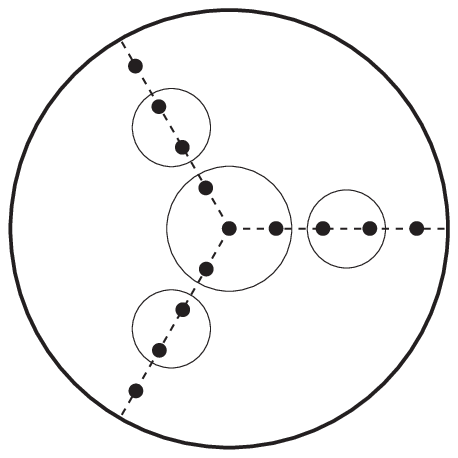}
    \ar[r]_{\eta_{m,\mathbf  d}}
    \ar[l]^{\phi_{m,d}}
&
\includegraphics[scale=.6]{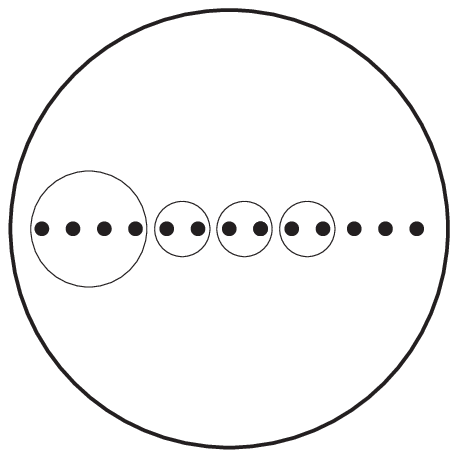}
}
$$
\hspace{6mm}$\C_\d$ in $D_{d+1}$ \hspace{20mm}
$\C_{m, \d}$ in $D_{n}^{(m)}$ \hspace{20mm}
$\C_{L_m(\d)}$ in $D_{n}$
\caption{The curve system $\C_\d$ in $D_{d+1}$ is lifted to
$\C_{m, \d}$ in $D_{n}^{(m)}$,
and then mapped to $\C_{L_m(\d)}$ in $D_{n}$.
Here $n=13$, $m=3$, $d=4$ and $\d=(2,2,1)$.}
\label{fig:cur-ray}
\end{figure}

We construct a diffeomorphism
$$
\eta_{m,\d}: D_n^{(m)}\to D_n
$$
similarly to $\eta_{m,d}$ such that
it fixes the boundary $\partial D_n^{(m)}=\partial D_n$ pointwise and
sends $\C_{m,\d}$ in $D_n^{(m)}$ to $\C_{L_m(\d)}$ in $D_n$.
First, we cut the punctured disk $D_n^{(m)}$ (resp.\ $D_n$)
by $r$ round circles disjoint from $\C_{m,\d}$ (resp.\ $\C_{L_m(\d)}$)
into one $(md_0+1)$-punctured disk and
$r$ annuli such that each annulus contains $m$ punctures
or $m$ components of $\C_{m,\d}$ (resp.\ $\C_{L_m(\d)}$)
as in Figure~\ref{fig:diffeo2}~(a).
Then, on each annulus, define $\eta_{m,\d}$
as in Figure~\ref{fig:diffeo2}~(b),
and on the $(md_0+1)$-punctured disk, define $\eta_{m,\d}$ as $\eta_{m,d_0}$.
Such a diffeomorphism $\eta_{m,\d}$
is unique up to isotopy.

\begin{figure}
\begin{tabular}{cc}
$\xymatrix{
\includegraphics[scale=.6]{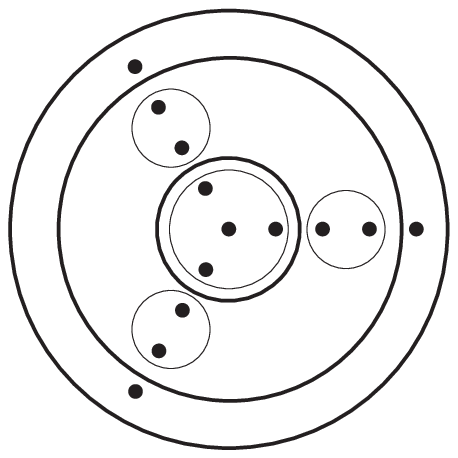}\ar[r]^{\eta_{m,\d}}&
\includegraphics[scale=.6]{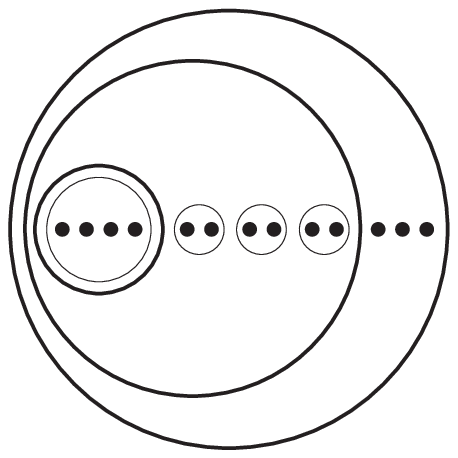}}$ &
$\xymatrix{
\includegraphics[scale=.6]{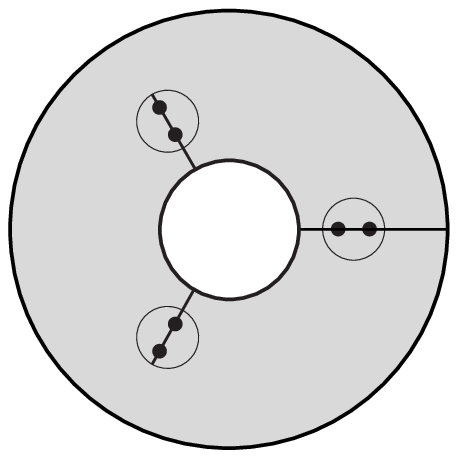}\ar[r]^{\eta_{m,\d}}&
\includegraphics[scale=.6]{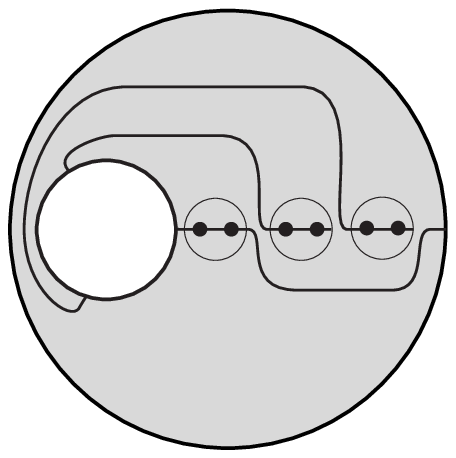}}$\\
(a) & (b)
\end{tabular}
\caption{The diffeomorphism $\eta_{m,\d}: D_{n}^{(m)} \to D_{n}$
where $n=13$, $m=3$, $d=4$ and $\d=(2,2,1)$.}
\label{fig:diffeo2}
\end{figure}

\begin{definition}\label{def:composition}
Let $n=md+1$ with $m\ge 2$ and $d\ge 1$.
Let $\d=(d_0+1,d_1,\ldots,d_r)$ be a composition of $d+1$ and $\n=L_m(\d)$.
We define an $n$-braid $\mu_{m,\mathbf d}$ as
$$
\mu_{m,\mathbf d}
=\myangle{\mu_{m,r}}_\n
(\mu_{m,d_0}\oplus(\underbrace{\Delta_{(d_1)}^2 \oplus 1 \oplus\cdots\oplus 1}_m )
\oplus\cdots\oplus
(\underbrace{\Delta_{(d_r)}^2 \oplus 1 \oplus\cdots\oplus 1}_m ))_\n.
$$
\end{definition}

See Figure~\ref{fig:mu-red}~(b) for $\mu_{m,\d}$ with $m=3$ and $\d=(2,2,1)$
and compare it with $\mu_{m,r}$ in Figure~\ref{fig:mu-red}~(a).
$\mu_{m,\d}*\C_{L_m(\d)}=\C_{L_m(\d)}$
and $\mu_{m,\d}$ is conjugate to $\epsilon_{(n)}^d$.
Observe that $\alpha\in B_n$ belongs to $Z(\mu_{m,\mathbf d})$
if and only if it is represented
by $\eta_{m,\mathbf d}\circ h\circ\eta_{m,\mathbf d}^{-1}$
for a $\rho_m$-equivariant diffeomorphism $h:D_n^{(m)}\to D_n^{(m)}$
which fixes the boundary pointwise.

\begin{figure}
\begin{tabular}{*5c}
\includegraphics[scale=.6]{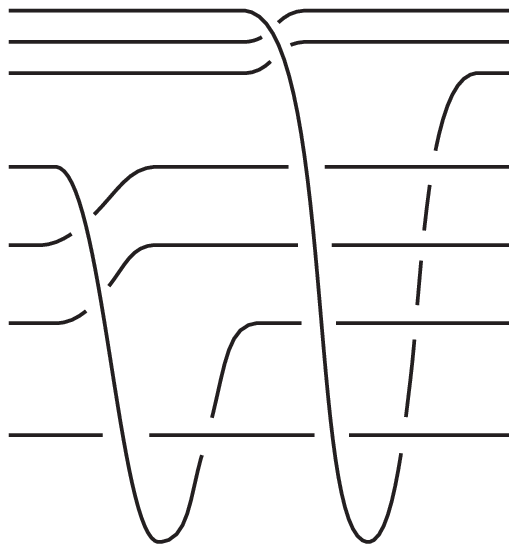} &~&
\includegraphics[scale=.6]{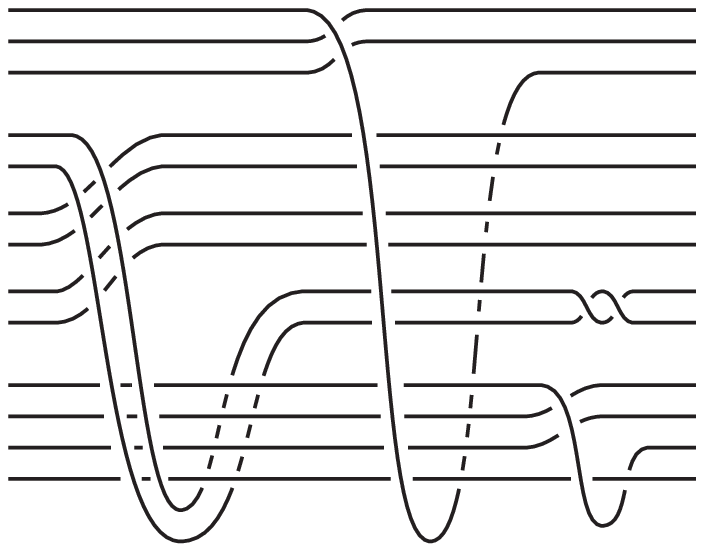} &~&
\includegraphics[scale=.6]{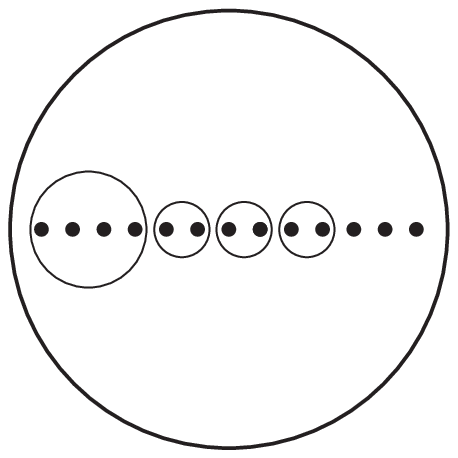} \\
(a) $\mu_{3,2}$ &&
(b) $\mu_{3,\d}$ &&
(c) $\C_{L_3(\d)}$
\end{tabular}
\caption{The braid $\mu_{3,\d}$ with $\d=(2,2,1)$ in (b)
is obtained from the braid $\mu_{3,2}$ in (a) by
taking four parallel copies of the first strand
and two parallel copies of the next three strands,
and then by concatenating some interior braids.
This braid has the standard reduction system
$\C_{L_3(\d)}=\C_{(4,2,2,2,1,1,1)}$ as in (c).}
\label{fig:mu-red}
\end{figure}

\begin{lemma}\label{lem:cent}
Let $n=md+1$ with $m\ge 2$ and $d\ge 1$.
Let $\d=(d_0+1,d_1,\ldots,d_r)$ be a composition of $d+1$, and let $\n=L_m(\d)$.
Let $\alpha$ be an $n$-braid with $\alpha*\C_\n=\C_\n$.
Then, $\alpha\in Z(\mu_{m,\d})$ if and only if\/
$\alpha$ is of the form
$$
\alpha=\myangle{\hat{\alpha}}_\n(\alpha_0\oplus
{\textstyle \bigoplus_{i=1}^r }
(\underbrace{\alpha_i\Delta_{(d_i)}^2\oplus\cdots\oplus\alpha_i\Delta_{(d_i)}^2}_{\ell_i}
\oplus \underbrace{\alpha_i\oplus\cdots\oplus\alpha_i}_{m-\ell_i})
)_\n
$$
such that $\hat{\alpha}\in Z(\mu_{m,r})$,
$\alpha_0\in Z(\mu_{m,d_0})$
and $d_{k_i}=d_i$ for $1\le i\le r$,
where $k_i$'s and $\ell_i$'s are integers such that
$\pi_{\hat\alpha}(x_{i,j})=x_{k_i, j-\ell_i}$ as in Lemma~\ref{lem:per}.
\end{lemma}

\begin{proof}
Since $\alpha*\C_\n=\C_\n$, $\alpha$ is of the form
$$
\alpha=\myangle{\hat{\alpha}}_\n
(\alpha_0\oplus
\underbrace{\alpha_{1,1}\oplus\cdots\oplus\alpha_{1,m}}_m
\oplus\cdots\oplus
\underbrace{\alpha_{r,1}\oplus\cdots\oplus\alpha_{r,m}}_m )_\n.
$$

Suppose that $\alpha\in Z(\mu_{m,\d})$.
Since $\alpha\ast\C_\n =\mu_{m,\d}\ast\C_\n = \C_\n$, we have
$$
\hat\alpha\mu_{m,r}=\Ext_{\n}(\alpha\mu_{m,\d})
=\Ext_{\n}(\mu_{m,\d}\alpha)=\mu_{m,r}\hat\alpha,
$$
hence $\hat\alpha\in Z(\mu_{m,r})$.
By Lemma~\ref{lem:per}, $\pi_{\hat\alpha}$ is of the form
$\pi_{\hat\alpha}(x_0)=x_0$ and
$\pi_{\hat\alpha}(x_{i,j})=x_{k_i, j-\ell_i}$,
where $1\le k_i\le r$ and $0\le \ell_i<m$.
Since $\pi_{\hat\alpha}(x_{i,j})=x_{k_i, j-\ell_i}$ and
$$
\n = L_m(\d)=(md_0+1,\underbrace{d_1,\ldots,d_1}_{m},\ldots,
\underbrace{d_r,\ldots,d_r}_{m}),
$$
one has
$$
\hat\alpha^{-1}*\n =
(md_0+1,\underbrace{d_{k_1},\ldots,d_{k_1}}_{m},\ldots,
\underbrace{d_{k_r},\ldots,d_{k_r}}_{m}).
$$
Therefore $d_{k_i}=d_i$ for $1\le i\le r$
because $\hat\alpha*\n=\n$.

Now we compute $\alpha\mu_{m,\d}$ and $\mu_{m,\d}\alpha$ straightforwardly:
\begin{eqnarray*}
\alpha\cdot \mu_{m,\d}
&=& \myangle{\hat{\alpha}}_\n (\alpha_0\oplus
{\textstyle \bigoplus_{i=1}^r}
(\alpha_{i,1}\oplus\cdots\oplus\alpha_{i,m}) )_\n \\
&& \mbox{}\quad \cdot
\myangle{\mu_{m,r}}_\n (\mu_{m,d_0}\oplus
{\textstyle \bigoplus_{i=1}^r}
(\Delta_{(d_i)}^2\oplus 1\oplus\cdots\oplus 1))_\n \\
&=&
\myangle{\hat{\alpha}}_\n \myangle{\mu_{m,r}}_\n
\cdot
(\alpha_0\oplus
{\textstyle \bigoplus_{i=1}^r}
(\alpha_{i,m}\oplus\alpha_{i,1}\oplus\cdots\oplus\alpha_{i,m-1}) )_\n\\
&& \mbox{}\quad \cdot
(\mu_{m,d_0}\oplus
{\textstyle \bigoplus_{i=1}^r}
(\Delta_{(d_i)}^2\oplus 1\oplus\cdots\oplus 1 ) )_\n \\
&=&
\myangle{\hat{\alpha}\mu_{m,r}}_\n
(\alpha_0\mu_{m,d_0}\oplus
{\textstyle \bigoplus_{i=1}^r}
(\alpha_{i,m}\Delta_{(d_i)}^2\oplus\alpha_{i,1}\oplus\cdots\oplus
\alpha_{i,m-1} ))_\n,\\
\mu_{m,\d}\cdot \alpha
&=& \myangle{\mu_{m,r}}_\n (\mu_{m,d_0}\oplus
{\textstyle \bigoplus_{i=1}^r}
(\Delta_{(d_i)}^2\oplus 1\oplus\cdots\oplus 1 ) )_\n\\
&&\mbox{}\qquad \cdot \myangle{\hat{\alpha}}_\n (\alpha_0\oplus
{\textstyle \bigoplus_{i=1}^r}
(\alpha_{i,1}\oplus\cdots\oplus\alpha_{i,m} )  )_\n \\
&=& \myangle{\mu_{m,r}}_\n \myangle{\hat{\alpha}}_\n \cdot (\mu_{m,d_0}\oplus
{\textstyle \bigoplus_{i=1}^r}
(\underbrace{1\oplus\cdots\oplus1}_{\ell_i}\oplus\Delta_{(d_i)}^2\oplus
 \underbrace{1\oplus\cdots\oplus1}_{m-\ell_i-1} ))_\n \\
&&\mbox{}\qquad \cdot (\alpha_0\oplus
{\textstyle \bigoplus_{i=1}^r}
(\alpha_{i,1}\oplus\cdots\oplus\alpha_{i,m} ) )_\n \\
&=&
\myangle{\mu_{m,r} \hat{\alpha}}_\n
(\mu_{m,d_0}\alpha_0\oplus\\
&&\mbox{}\qquad
{\textstyle \bigoplus_{i=1}^r}
(\alpha_{i,1}\oplus\cdots\oplus\alpha_{i,\ell_i}\oplus
    \Delta_{(d_i)}^2\alpha_{i,\ell_i+1}\oplus\alpha_{i,\ell_i+2}
    \oplus\cdots\oplus\alpha_{i,m}))_\n.
\end{eqnarray*}
Comparing the above two formulae, we have
$\alpha_0\in Z(\mu_{m,d_0})$.
In addition, for $1\le i\le r$,
\begin{eqnarray*}
&& \alpha_{i,m}\Delta_{(d_i)}^{2} = \alpha_{i,1} = \alpha_{i,2} = \cdots = \alpha_{i,\ell_i} =
\Delta_{(d_i)}^{2}\alpha_{i,\ell_i +1}, \\
&& \alpha_{i,\ell_i +1} = \alpha_{i,\ell_i +2} = \cdots = \alpha_{i,m}.
\end{eqnarray*}
Let $\alpha_i =\alpha_{i,m}$, then $\alpha$ has the desired form.
The other direction is straightforward.
\end{proof}

\begin{corollary}\label{cor:cent}
Let $n=md+1$ with $m\ge 2$ and $d\ge 1$.
Let $\d=(d_0+1,d_1,\ldots,d_r)$ be a composition of $d+1$,
and let $\n=L_m(\d)$.
Let $\alpha$ be an $n$-braid with $\alpha*\C_\n=\C_\n$,
and let the $\n$-exterior braid of $\alpha$ be pure.
Then, $\alpha\in Z(\mu_{m,\d})$ if and only if
$\alpha$ is of the form
$$
\alpha=\myangle{\hat{\alpha}}_\n
(\alpha_0\oplus m\alpha_1\oplus\cdots\oplus m\alpha_r)_\n
$$
where $\hat{\alpha}\in Z(\mu_{m,r})$ and $\alpha_0\in Z(\mu_{m,d_0})$.
\end{corollary}

\begin{proof}
Since $\Ext_\n(\alpha)$ is pure,
$k_i=i$ and $\ell_i=0$ for $i=1,\ldots,r$ in Lemma~\ref{lem:cent}.
\end{proof}

\begin{lemma}\label{lem:sta}
Let $n=md+1$ with $m\ge 2$ and $d\ge 1$.
Let $\C$ be an unnested curve system in $D_n$ such that $\mu_{m,d}*\C=\C$.
Then there exist $\zeta\in Z(\mu_{m,d})$ and a composition
$\d=(d_0+1,d_1,\ldots,d_r)$ of\/ $d+1$ with $d_1\ge \cdots \ge d_r$ such that
$$
(\chi\zeta)*\C=\C_{L_m(\mathbf d)},
$$
where $\chi$ is an $n$-braid represented by $\eta_{m,\d}\circ\eta_{m,d}^{-1}$.
\end{lemma}

\begin{proof}
The construction of $\d$ and $\zeta$ can be viewed in
Figure~\ref{fig:st}.
Let $\C'=\eta_{m,d}^{-1}(\C)$.
It is a curve system in $D_n^{(m)}$
such that $\rho_m (\C')$ is isotopic to $\C'$.

\begin{figure}
$
\xymatrix{
\hbox to 0pt{\hskip 1mm \fbox{$\C$}\hss}
\includegraphics[scale=.65]{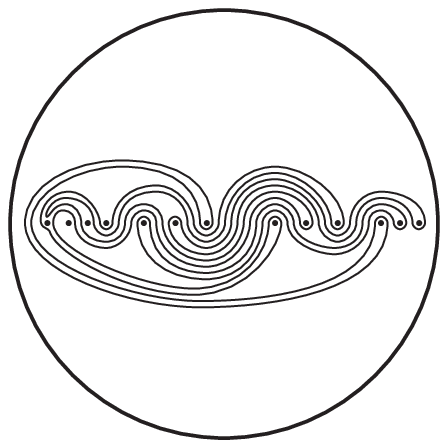} \ar@{.>}[r]^{h} &
\includegraphics[scale=.65]{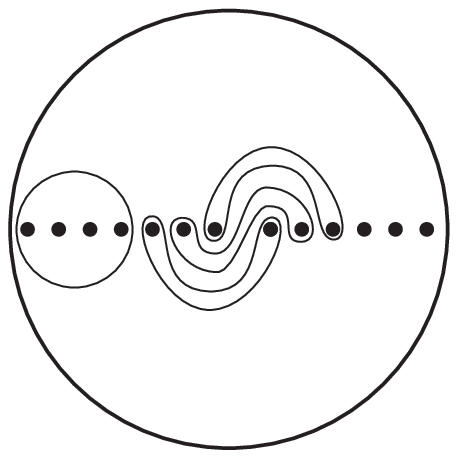} \ar@{.>}[r]^f &
\includegraphics[scale=.65]{sym-St13.eps}
\hbox to 0pt{\hskip -7mm\fbox{$\C_{L_m(\d)}$}\hss}\\
\hbox to 0pt{\hskip 1mm \fbox{$\C'$}\hss}
\includegraphics[scale=.65]{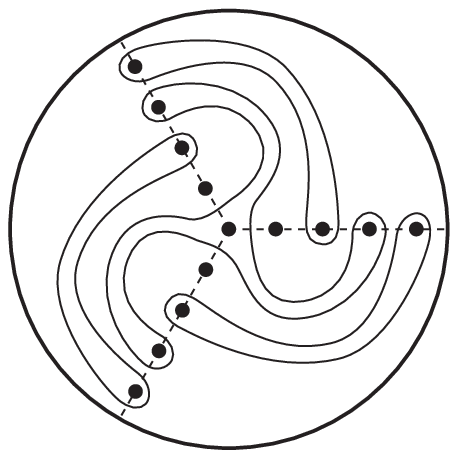} \ar@{.>}[r]^g
  \ar[d]_{\phi_{m,d}}
  \ar[u]_{\eta_{m,d}} &
\includegraphics[scale=.65]{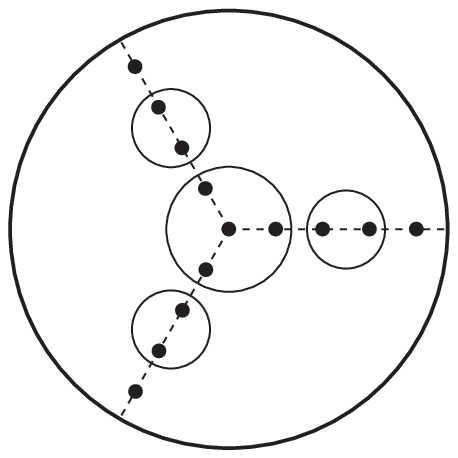}
\hbox to 0pt{\hskip -7mm\fbox{$\C_{m,\d}$}\hss}
  \ar[d]_{\phi_{m,d}}
  \ar[u]_{\eta_{m,d}}
  \ar[ur]_{\eta_{m,\d}} \\
\hbox to 0pt{\hskip 1mm \fbox{$\bar\C'$}\hss}
\includegraphics[scale=.65]{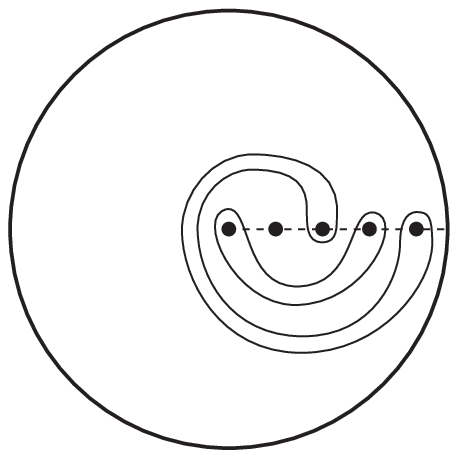} \ar@{.>}[r]^{\bar g}  &
\includegraphics[scale=.65]{sym-disc-b1.eps}
\hbox to 0pt{\hskip -7mm\fbox{$\C_\d$}\hss}\\
}
$
\caption{This figure shows the constructions
in the proof of Lemma~\ref{lem:sta}.}\label{fig:st}
\end{figure}

\begin{claim}{Claim}
\emph{Without loss of generality, we may assume that $\rho_m(\C')=\C'$.}
\end{claim}

\begin{proof}
Give a complete hyperbolic geometry on the interior of $D_{d+1}$,
and lift it to the interior of $D_n^{(m)}$ by using the covering
$\phi_{m,d}:D_n^{(m)}\to D_{d+1}$.
Then $\rho_m$ is an isometry.
Let $\C''$ be the unique geodesic curve system isotopic to $\C'$.
Because $\rho_m(\C'')$ is isotopic to $\C''$,
we have $\rho_m(\C'')=\C''$.
\end{proof}

Now we assume that $\rho_m (\C')=\C'$.
Let $\bar\C'=\phi_{m,d}(\C')$.
Because $\bar\C'$ is an unnested curve system in $D_{d+1}$,
there exists a diffeomorphism $\bar g:D_{d+1}\to D_{d+1}$
fixing the first puncture and the points at $\partial D_{d+1}$
such that $\bar g(\bar \C')$ is a standard unnested curve system,
hence $\bar g(\bar \C')=\C_\d$ for a composition
$\d=(d_0+1,d_1,\ldots,d_r)$ of $d+1$.
Note that $\bar g$ can be chosen so that $d_1\ge d_2\ge\cdots\ge d_r$.

Let $g:D_n^{(m)}\to D_n^{(m)}$ be the lift of $\bar{g}$
that fixes $\partial D_n^{(m)}$ pointwise.
Then $g$ is a $\rho_m$-equivariant diffeomorphism with
$g(\C')=\phi_{m,d}^{-1}(\C_\d)=\C_{m,\d}$.
Let $h=\eta_{m,d}\circ g\circ \eta_{m,d}^{-1}$
and $f=\eta_{m,\d}\circ \eta_{m,d}^{-1}$.
Let $\zeta$ and $\chi$ be $n$-braids represented by $h$ and $f$, respectively.
Because $g$ is $\rho_m$-equivariant, $\zeta\in Z(\mu_{m,d})$.
By the construction, $(f\circ h)(\C)=\C_{L_m(\d)}$,
hence $(\chi\zeta)*\C=\C_{L_m(\d)}$.
\end{proof}

\begin{corollary}\label{cor:sta}
Let $n=md+1$ with $m\ge 2$ and $d\ge 1$.
Let $\C_1$ and $\C_2$ be unnested curve systems in $D_n$
that are of the same type and invariant under the action of\/ $\mu_{m,d}$.
Then there exist $\chi\in B_n$, $\chi_1,\chi_2\in Z(\mu_{m,d})$ and
a composition $\d=(d_0+1,d_1,\ldots,d_r)$ of\/ $d+1$
such that
$$
\chi\mu_{m,d}\chi^{-1}=\mu_{m,\d}
\quad\mbox{and}\quad
(\chi\chi_1)*\C_1=(\chi\chi_2)*\C_2=\C_{L_m(\mathbf d)}.
$$
\end{corollary}

\begin{proof}
Applying Lemma~\ref{lem:sta} to $\C_1$ and $\C_2$,
we have $\chi_1,\chi_2\in Z(\mu_{m,d})$ and compositions
$\d=(d_0+1,d_1,\ldots,d_r)$ and $\d'=(d_0'+1,d_1',\ldots,d_{r'}')$ of $d+1$
with $d_1\ge\cdots\ge d_r$ and $d_1'\ge\cdots\ge d_{r'}'$
such that
$$
(\chi\chi_1)*\C_1=\C_{L_m(\d)}
\quad\mbox{and}\quad
(\chi'\chi_2)*\C_2=\C_{L_m(\d')},
$$
where $\chi$ and $\chi'$ are represented by
$\eta_{m,\d}\circ\eta_{m,d}^{-1}$ and
$\eta_{m,\d'}\circ\eta_{m,d}^{-1}$, respectively.
Because $\C_1$ and $\C_2$ are of the same type, so are
$\C_{L_m(\d)}$ and $\C_{L_m(\d')}$.
This implies $\d=\d'$ because
$d_1\ge\cdots\ge d_r$ and $d_1'\ge\cdots\ge d_{r'}'$.
Therefore $\chi=\chi'$.

Notice that the diffeomorphisms
$(\eta_{m,\d}\circ\eta_{m,d}^{-1})\circ (\eta_{m,d}\circ \rho_m \circ\eta_{m,d}^{-1})
\circ (\eta_{m,\d}\circ\eta_{m,d}^{-1})^{-1}$
and $\eta_{m,\d}\circ\rho_m\circ\eta_{m,\d}^{-1}$
represent $\chi\mu_{m,d}\chi^{-1}$ and $\mu_{m,\d}$, respectively, modulo $\Delta^2$.
Since the two diffeomorphisms are the same,
$\chi\mu_{m,d}\chi^{-1}=\mu_{m,\d} \bmod \Delta^2$.
Therefore
$$
\chi\mu_{m,d}\chi^{-1}=\mu_{m,\d}
$$
because $\chi\mu_{m,d}\chi^{-1}$ and $\mu_{m,\d}$ are conjugate.
\end{proof}

\begin{definition}
Let $\Lambda$ be a choice function for the set $\cup_{k\ge 1} B_k$,
hence $\Lambda(A)\in A$ for $A\subset\cup_{k\ge 1} B_k$.
For $\alpha\in B_k$, let $[\alpha]$ denote the conjugacy class of $\alpha$ in $B_k$.
Let $m\ge 2$ and $t \ge 0$.
We define
$$
\lambda=\lambda_{m,t}:\cup_{k\ge 1} B_k\to \cup_{k\ge 1} B_k
$$
as $\lambda(\alpha)=\Lambda(Z(\mu_{m,t})\cap [\alpha])$
if $\alpha\in B_{mt+1}$ and $[\alpha]\cap Z(\mu_{m,t})\ne\emptyset$, and
$\lambda(\alpha)=\Lambda([\alpha])$ otherwise.
We call $\lambda(\alpha)$ the \emph{conjugacy representative} of $\alpha$
with respect to $(m,t)$.
\end{definition}

Then $\lambda$ has the following properties.
\begin{itemize}
\item[(i)]
$\lambda(\alpha)$ is conjugate to $\alpha$.
Two braids $\alpha$ and $\beta$ are conjugate
if and only if $\lambda(\alpha)=\lambda(\beta)$.

\item[(ii)]
If $\alpha\in Z(\mu_{m,t})$, then $\lambda(\alpha)\in Z(\mu_{m,t})$.
\end{itemize}

\begin{lemma}\label{lem:choice}
Let\/ $n=md+1$ with $m\ge 2$ and $d\ge 1$.
Let\/ $\d=(d_0+1,d_1,\ldots,d_r)$ be a composition of\/ $d+1$,
and let $\n=L_m(\d)$.
Let
$\alpha = \myangle{\hat{\alpha}}_\n (\alpha_0\oplus
m\alpha_1\oplus\cdots\oplus m\alpha_r)_\n$
belong to $Z(\mu_{m,\d})$ such that $\hat\alpha$ is a pure braid.
Let $\lambda=\lambda_{m,d_0}$.
Suppose that $\alpha_0$ and $\lambda (\alpha_0)$ are conjugate in $Z(\mu_{m, d_0})$.
Then there exists an $\n$-split braid $\gamma$ in $Z(\mu_{m,\d})$
such that
$$
\gamma\alpha\gamma^{-1}
= \myangle{\hat{\alpha}}_\n (\lambda(\alpha_0)\oplus
m \lambda(\alpha_1)\oplus\cdots\oplus m\lambda(\alpha_r) )_\n.
$$
\end{lemma}

\begin{proof}
Since $\alpha\in Z(\mu_{m,\d})$, one has $\alpha_0\in Z(\mu_{m,d_0})$ by Lemma~\ref{lem:cent},
hence $\lambda(\alpha_0)\in Z(\mu_{m,d_0})$.
By the hypothesis,
there exists $\gamma_0\in Z(\mu_{m,d_0})$ such that
$\gamma_0\alpha_0\gamma_0^{-1}=\lambda(\alpha_0)$.
For $1\le i\le r$, choose $\gamma_i\in B_{d_i}$ such that
$\gamma_i\alpha_i\gamma_i^{-1}=\lambda(\alpha_i)$.
Using $\gamma_i$'s, define an $\n$-split braid $\gamma$ as
$$
\gamma = \myangle{1}_\n (\gamma_0\oplus
m\gamma_1\oplus\cdots\oplus m\gamma_r )_\n.
$$
By Corollary~\ref{cor:cent}, $\gamma\in Z(\mu_{m,\d})$.
Because the $\n$-exterior braids of $\alpha$ and $\gamma$ are pure braids,
we can compute the conjugation $\gamma\alpha\gamma^{-1}$ component-wise as follows.
\begin{eqnarray*}
\gamma\alpha\gamma^{-1}
&=& \myangle{1}_\n \myangle{\hat{\alpha}}_\n \myangle{1}_\n^{-1}
(\gamma_0 \alpha_0\gamma_0^{-1}\oplus
m(\gamma_1 \alpha_1\gamma_1^{-1})\oplus \cdots\oplus
m(\gamma_r \alpha_r\gamma_r^{-1}) )_\n\\
&=& \myangle{\hat{\alpha}}_\n (\lambda(\alpha_0)\oplus
m\lambda(\alpha_1)\oplus\cdots\oplus
m\lambda(\alpha_r) )_\n.
\end{eqnarray*}
This completes the proof.
\end{proof}

\section{Proof of Theorem~\ref{thm:conj}}
\label{sec:conj}

In this section we prove Theorem~\ref{thm:conj}.
Let us explain our strategy.

Let $\omega$ be a periodic $n$-braid, and let $\alpha,\beta\in Z(\omega)$
be conjugate in $B_n$.
If $\alpha$ and $\beta$ are periodic,
it is easy to prove the theorem because we may assume
that $\alpha$ and $\beta$ are central by~Corollary~\ref{cor:unique}.
If $\alpha$ and $\beta$ are pseudo-Anosov, then
we obtain a stronger result in Proposition~\ref{prop:pAconj} that
any conjugating element between $\alpha$ and $\beta$ must be
contained in $Z(\omega)$.
(This result will be crucial in dealing with
reducible braids.)
However a far more laborious task is required
when $\alpha$ and $\beta$ are reducible.

\medskip

Let us say that a periodic braid is \emph{$\delta$-type}
(resp.\ \emph{$\epsilon$-type}) if it is conjugate to
a power of $\delta$ (resp.\ a power of $\epsilon$).
The periodic braid $\omega$ is either $\delta$-type or $\epsilon$-type,
and Theorem~\ref{thm:conj} can be proved similarly for the two cases.
One possible strategy for escaping repetition of similar arguments
would be writing an explicit proof for one case
and then showing that the other case follows easily from that.
However this strategy hardly works.
It seems that (when $\alpha$ and $\beta$ are reducible)
the proof for $\delta$-type $\omega$
must contain most of the whole arguments needed in the proof
for $\epsilon$-type $\omega$,
and that the theorem for $\delta$-type $\omega$
is not a consequence of that for $\epsilon$-type $\omega$.

Proposition~\ref{prop:key} is a key step in the proof of the theorem
when $\alpha$ and $\beta$ are reducible.
The proposition may look rather unnatural at a glance,
but it is written in that form
in order to minimize repetition of similar arguments.

\begin{proposition}\label{prop:pAconj}
Let $\omega$ be a periodic $n$-braid,
and let $\alpha$ and $\beta$ be pseudo-Anosov braids in $Z(\omega)$.
If $\gamma$ is an $n$-braid with $\beta=\gamma\alpha\gamma^{-1}$,
then $\gamma\in Z(\omega)$.
\end{proposition}

\begin{proof}
Let $I_\omega:B_n\to B_n$ be the inner automorphism such that
$I_\omega(\chi)=\omega^{-1}\chi\omega$ for $\chi\in B_n$.
Since $\alpha,\beta\in Z(\omega)$,
we have $I_\omega(\alpha)=\alpha$ and $I_\omega(\beta)=\beta$.
Therefore
$$
\gamma\alpha\gamma^{-1}
=\beta=I_\omega(\beta)=I_\omega(\gamma\alpha\gamma^{-1})
=I_\omega(\gamma) I_\omega(\alpha)I_\omega(\gamma^{-1})
=I_\omega(\gamma)\alpha I_\omega(\gamma)^{-1}.
$$
Let $\gamma_0=\gamma^{-1}I_\omega(\gamma)$.
Then $\gamma_0\in Z(\alpha)$.
It is known by J.~Gonz\'alez-Meneses and B.~Wiest~\cite{GW04}
that the centralizer of a pseudo-Anosov braid is a free abelian group of rank two.
Since $\alpha$ is pseudo-Anosov and $\gamma_0,\omega\in Z(\alpha)$,
$\gamma_0$ commutes with $\omega$,
hence $I_\omega^k(\gamma_0)=\gamma_0$ for all $k$.
Take a positive integer $m$ such that $\omega^m$ is central.
Since $I_\omega^m(\gamma)=\gamma$,
\begin{eqnarray*}
\gamma_0^m
&=&\gamma_0\gamma_0\cdots\gamma_0
=\gamma_0\cdot I_\omega(\gamma_0)\cdots I_\omega^{m-1}(\gamma_0)\\
&=&\gamma^{-1}I_\omega(\gamma)\cdot I_\omega(\gamma^{-1})I_\omega^2(\gamma)
\cdots I_\omega^{m-1}(\gamma^{-1})I_\omega^m(\gamma)\\
&=&\gamma^{-1}I_\omega^m(\gamma)=\gamma^{-1}\gamma=1.
\end{eqnarray*}
Since braid groups are torsion free
(for example see~\cite{Deh98}),
we have $\gamma_0=1$, hence $\gamma=I_\omega(\gamma)$.
This means that $\gamma\in Z(\omega)$.
\end{proof}

We remark that, in the above proposition,
the pseudo-Anosov condition on $\alpha$ and $\beta$ is necessary.
Let $\omega$ be a non-central periodic braid,
and let $\gamma$ be a braid which does not belong to $Z(\omega)$.
Let $\alpha=\beta=\Delta^2$.
Because $\Delta^2$ is central,
we have $\alpha,\beta\in Z(\omega)$ and
$\gamma\alpha\gamma^{-1}=\beta$.
However, $\gamma\not\in Z(\omega)$.

\medskip

The following proposition is a key step in the proof of
Theorem~\ref{thm:conj}.
Recall that $\nu:B_{n,1}\to B_{n-1}$ is the homomorphism
deleting the first strand.

\begin{proposition}\label{prop:key}
Let $m\ge 2$ and $d\ge 1$.
Let $\omega$ be a 1-pure braid conjugate to $\epsilon_{(md+1)}^d$.
Let $\alpha$ and $\beta$ be pure braids in $Z(\omega)$ such that either
$\alpha$ and $\beta$ are conjugate in $B_{md+1}$ or
$\nu(\alpha)$ and $\nu(\beta)$ are conjugate in $B_{md}$.
Then $\alpha$ and $\beta$ are conjugate in $Z(\omega)$.
\end{proposition}

Before proving the above proposition, we prove
Theorem~\ref{thm:conj} assuming the proposition.

\begin{proof}[Proof of Theorem~\ref{thm:conj}]
Let $\omega$ be a periodic $n$-braid,
and let $\alpha, \beta\in Z(\omega)$ be conjugate in $B_{n}$.
We will show that $\alpha$ and $\beta$ are conjugate in $Z(\omega)$.

\begin{claim}{Claim}
Without loss of generality,
we may assume that $\omega$ is non-central, and both $\alpha$ and $\beta$ are pure braids.
\end{claim}

\begin{proof}[Proof of Claim]
If $\omega$ is central,
then $Z(\omega)=B_n$, hence there is nothing to prove.
Hence we may assume that $\omega$ is non-central.

Assume that Theorem~\ref{thm:conj} holds for pure braids.
Choose a positive integer $k$ such that $\alpha^k$ and $\beta^k$ are pure.
By assumption, $\alpha^k$ and $\beta^k$ are conjugate in $Z(\omega)$.
Because $Z(\omega)$ is isomorphic to $A(\arB_\ell)$ for some $\ell\ge 1$,
$\alpha$ and $\beta$ are conjugate in $Z(\omega)$
by Corollary~\ref{cor:unique}.
This shows that it suffices to prove Theorem~\ref{thm:conj}
for the case where $\alpha$ and $\beta$ are pure braids.
\end{proof}

Now we assume that $\omega$ is noncentral,
and both $\alpha$ and $\beta$ are pure braids.

By Lemma~\ref{lem:easy}~(i), we may assume
that $\omega=\delta^d$ or $\omega=\epsilon^d$ for some $d$.
By Lemma~\ref{lem:order}, we may further assume that $d$
is a positive integer which is a divisor
of $n$ if $\omega=\delta^d$ and a divisor of $n-1$ if $\omega=\epsilon^d$.
Therefore there is a positive integer $m$ such that
$n=md$ if $\omega=\delta^d$ and $n=md+1$ if $\omega=\epsilon^d$.
Note that $m\ge 2$ because $\omega$ is central if $m=1$.

If $\omega=\epsilon^d$,
then $\alpha$ and $\beta$ are conjugate in $Z(\omega)$ by Proposition~\ref{prop:key}.

Now, suppose that $\omega =\delta^d$.
Let $\bar\alpha$ and $\bar\beta$ be the elements of $B_{d+1,1}$
such that $\psi_5(\bar\alpha)=\alpha$ and $\psi_5(\bar\beta)=\beta$.
Let $\alpha'=\psi_4(\bar\alpha)$ and $\beta'=\psi_4(\bar\beta)$.
Then $\alpha'$ and $\beta'$ are $(n+1)$-braids in $Z(\epsilon_{(n+1)}^d)$
such that $\nu(\alpha')=\alpha$ and $\nu(\beta')=\beta$ are conjugate in $B_{n}$.
$$
\xymatrix{
B_n \supset Z(\delta_{(n)}^d)
&&
Z(\epsilon_{(n+1)}^d) \subset B_{n+1,1} \ar[ll]_\nu \\
&
B_{d+1,1} \ar[ul]^{\psi_5} \ar[ur]_{\psi_4}
}\qquad
\xymatrix{
\alpha,~\beta && \alpha',~\beta' \ar@{|->}[ll]_\nu \\
&
\bar\alpha,~\bar\beta \ar@{|->}[ul]^{\psi_5} \ar@{|->}[ur]_{\psi_4}
}
$$

Since $\alpha$ and $\beta$ are pure, so are $\alpha'$ and $\beta'$.
Applying Proposition~\ref{prop:key} to $(\alpha',\beta',\epsilon_{(n+1)}^d)$,
we conclude that $\alpha'$ and $\beta'$ are conjugate in $Z(\epsilon_{(n+1)}^d)$.
This implies that $\bar\alpha$ and $\bar\beta$ are conjugate in $B_{d+1,1}$,
hence $\alpha$ and $\beta$ are conjugate in $Z(\delta^d)=Z(\omega)$.
\end{proof}

\begin{proof}[Proof of Proposition~\ref{prop:key}]
Let $n=md+1$.
Recall that $m\ge 2$ and $d\ge 1$;
$\omega$ is a 1-pure braid conjugate to $\epsilon_{(n)}^d$;
$\alpha$ and $\beta$ are pure braids in $Z(\omega)$ such that either
$\alpha$ and $\beta$ are conjugate in $B_{n}$ or
$\nu(\alpha)$ and $\nu(\beta)$ are conjugate in $B_{n-1}$.
We will show that $\alpha$ and $\beta$ are conjugate in $Z(\omega)$
by using induction on $d\ge 1$.
By Lemma~\ref{lem:easy}, we may assume that $\omega=\mu_{m,d}$.

\begin{claim}{Claim 1}
\begin{itemize}
\item[(i)]
$\alpha$ and $\beta$ are conjugate in $Z(\omega)$
if and only if $\nu(\alpha)$ and $\nu(\beta)$ are conjugate in $Z(\nu(\omega))$.
\item[(ii)]
The braids $\alpha$, $\beta$, $\nu(\alpha)$ and $\nu(\beta)$
have the same Nielsen-Thurston type.
\item[(iii)]
The proposition holds if\/
$\alpha$ is periodic or pseudo-Anosov, or if\/ $d=1$.
\item[(iv)]
If\/ $\alpha$ is non-periodic and reducible,
then $\Rext(\alpha)$ and $\Rext(\beta)$ are of the same type.
\end{itemize}
\end{claim}

\begin{proof}
Let $\bar\alpha$ and $\bar\beta$ be the elements of $B_{d+1,1}$ such that
$\alpha=\psi_4'(\bar\alpha)$ and $\beta=\psi_4'(\bar\beta)$,
hence $\nu(\alpha)=\psi_5'(\bar\alpha)$ and $\nu(\beta)=\psi_5'(\bar\beta)$.
We have the following commutative diagram, where the isomorphisms
preserve the Nielsen-Thurston type
by Lemma~\ref{lem:dyn} and Corollary~\ref{cor:dyn}.
$$
\xymatrix{
Z(\nu(\omega))
&&
Z(\omega) \ar[ll]_\nu^\simeq \\
&
B_{d+1,1} \ar[ul]^{\psi_5'}_\simeq \ar[ur]_{\psi_4'}^\simeq
}\qquad
\xymatrix{
\nu(\alpha),~\nu(\beta) && \alpha,~\beta \ar@{|->}[ll]_\nu \\
&
\bar\alpha,~\bar\beta \ar@{|->}[ul]^{\psi_5'} \ar@{|->}[ur]_{\psi_4'}
}
$$

\smallskip(i)\ \
This follows from the fact that $\nu:Z(\omega)\to Z(\nu(\omega))$ is an isomorphism.

\smallskip(ii)\ \
Because $\nu$ preserves the Nielsen-Thurston type,
$(\alpha, \nu(\alpha))$ and $(\beta,\nu(\beta))$ are pairs of braids
with the same Nielsen-Thurston type.
By the hypothesis, either $(\alpha,\beta)$ or $(\nu(\alpha),\nu(\beta))$
is a pair of braids with the same Nielsen-Thurston type.
Therefore the braids $\alpha$, $\beta$, $\nu(\alpha)$ and $\nu(\beta)$ have
the same Nielsen-Thurston type.

\smallskip(iii)\ \
Suppose that $\alpha$ is pseudo-Anosov.
Then $\beta$, $\nu(\alpha)$ and $\nu(\beta)$ are pseudo-Anosov by (ii).
If $\alpha$ and $\beta$ are conjugate in $B_{n}$,
they are conjugate in $Z(\omega)$ by Proposition~\ref{prop:pAconj}.
If $\nu(\alpha)$ and $\nu(\beta)$ are conjugate in $B_{n-1}$,
they are conjugate in $Z(\nu(\omega))$ by Proposition~\ref{prop:pAconj},
hence $\alpha$ and $\beta$ are conjugate in $Z(\omega)$ by (i).

Suppose that $\alpha$ is periodic.
Then $\beta$, $\nu(\alpha)$ and $\nu(\beta)$ are periodic by (ii).
Since $\alpha$, $\beta$, $\nu(\alpha)$ and $\nu(\beta)$
are periodic and pure, they are central.
If $\alpha$ and $\beta$ are conjugate in $B_{n}$,
then $\alpha =\beta$.
If $\nu(\alpha)$ and $\nu(\beta)$ are conjugate in $B_{n-1}$,
then $\nu(\alpha) = \nu(\beta)$, hence $\alpha =\beta$.

Suppose that $d=1$.
Because $\bar\alpha\in B_{2,1}=\{\sigma_1^{2k}\mid k\in\Z\}=\{\Delta_{(2)}^{2k}\mid k\in\Z\}$,
$\bar\alpha$ is periodic.
Hence $\alpha$ is periodic.
From the above discussion, $\alpha$ and $\beta$ are conjugate in $Z(\omega)$.

\smallskip(iv)\ \
The braids $\nu(\alpha)$, $\beta$ and $\nu(\beta)$ are
non-periodic and reducible by~(ii).

If $\alpha$ and $\beta$ are conjugate in $B_n$, then
$\Rext(\alpha)$ and $\Rext(\beta)$ are of the same type.

Suppose that $\nu(\alpha)$ and $\nu(\beta)$ are conjugate in $B_{n-1}$.
Taking conjugates of $\bar\alpha$ and $\bar\beta$
by elements of $B_{d+1,1}$ if necessary, we may assume that
$\Rext(\bar\alpha)$ and $\Rext(\bar\beta)$ are standard.
Let $\d_1 = (d_0 +1, d_1,\ldots, d_r)$ and $\d_2 = (e_0 +1, e_1,\ldots, e_s)$
be the compositions of $d+1$ such that $\Rext(\bar\alpha)=\C_{\d_1}$
and $\Rext(\bar\beta)=\C_{\d_2}$.
Let
$$\arraycolsep=2pt
\begin{array}{*7{l}}
\n_1 & = & ( md_0+1, \underbrace{d_1,\ldots, d_1}_{m}, \ldots,
\underbrace{d_r,\ldots, d_r}_{m}), &&
\n_1' & =& ( md_0, \underbrace{d_1,\ldots, d_1}_{m}, \ldots,
\underbrace{d_r,\ldots, d_r}_{m}), \\
\n_2 &=& ( me_0+1, \underbrace{e_1,\ldots, e_1}_{m}, \ldots,
\underbrace{e_s,\ldots, e_s}_{m}), &&
\n_2' &=& ( me_0, \underbrace{e_1,\ldots, e_1}_{m}, \ldots,
\underbrace{e_s,\ldots, e_s}_{m}),
\end{array}$$
where $md_0$ (resp.\ $me_0$) is deleted if $d_0 =0$ (resp. $e_0 =0$).
Let us write ``$\C_1\approx\C_2$''
if two curve systems $\C_1$ and $\C_2$ are of the same type.
By Corollary~\ref{cor:crs},
$$
\Rext(\alpha)\approx\C_{\n_1},\quad
\Rext(\beta)\approx\C_{\n_2},\quad
\Rext(\nu(\alpha))\approx\C_{\n_1'},\quad
\Rext(\nu(\beta))\approx\C_{\n_2'}.
$$
Since $\nu(\alpha)$ and $\nu(\beta)$ are conjugate in $B_{n-1}$,
one has $\Rext(\nu(\alpha))\approx\Rext(\nu(\beta))$,
and hence $\C_{\n_1'}\approx\C_{\n_2'}$.
Thus the following equality holds between multisets consisting of integers:
$$
\{ md_0, \underbrace{d_1,\ldots, d_1}_{m}, \ldots,
\underbrace{d_r,\ldots, d_r}_{m}\}
= \{ me_0, \underbrace{e_1,\ldots, e_1}_{m}, \ldots,
\underbrace{e_s,\ldots, e_s}_{m}\}.
$$
Since $m\ge 2$, one has $r=s$ and $d_0=e_0$ and there is an $r$-permutation $\theta$
such that $d_i=e_{\theta(i)}$ for $1\le i\le r$.
Therefore $\C_{\n_1}$ and $\C_{\n_2}$
are of the same type, hence so are $\Rext(\alpha)$ and $\Rext(\beta)$.
\end{proof}

By the above claim, the proposition holds if $d=1$,
or if $\alpha$ is periodic or pseudo-Anosov.
Thus we may assume that $d\ge 2$ and
that $\alpha$ is non-periodic and reducible.

\begin{claim}{Claim 2}
Without loss of generality, we may assume that
$\omega = \mu_{m,\d}$ for a composition $\d=(d_0+1,d_1,\ldots,d_r)$ of\/ $d+1$,
and that $\Rext(\alpha)=\Rext(\beta)=\C_{\n}$ where $\n= L_m(\mathbf d)$.
\end{claim}

\begin{proof}[Proof of Claim 2]
Recall that we have assumed that $\omega=\mu_{m,d}$.
By Claim~1~(iv), $\Rext(\alpha)$ and $\Rext(\beta)$ are of the same type.
They are invariant under the action of $\omega$.
By Corollary~\ref{cor:sta},
there exist $\chi\in B_n$, $\chi_1,\chi_2\in Z(\omega)$
and a composition $\d$ of $d+1$ such that
$\chi\omega\chi^{-1}=\mu_{m,\d}$ and
$(\chi\chi_1)*\Rext(\alpha) = (\chi\chi_2)*\Rext(\beta) = \C_{L_m(\d)}$.
Notice that $(\chi\chi_1)*\Rext(\alpha) = \Rext(\chi\chi_1\alpha\chi_1^{-1}\chi^{-1})$ and
$(\chi\chi_2)*\Rext(\beta) = \Rext(\chi\chi_2\beta\chi_2^{-1}\chi^{-1})$.
By Lemma~\ref{lem:easy}, it suffices to show that
$\chi\chi_1\alpha\chi_1^{-1}\chi^{-1}$
and $\chi\chi_2\beta\chi_2^{-1}\chi^{-1}$ are conjugate
in $Z(\chi\omega\chi^{-1})=Z(\mu_{m,\d})$.
\end{proof}

Let us assume the hypothesis stated in Claim 2.
By Corollary~\ref{cor:cent},
$\alpha$ and $\beta$ are written as
\begin{eqnarray*}
\alpha
&=& \myangle{\hat{\alpha}}_\n (\alpha_0\oplus m\alpha_1\oplus\cdots\oplus m\alpha_r)_\n,\\
\beta
&=& \myangle{\hat{\beta}}_\n(\beta_0\oplus m\beta_1\oplus\cdots\oplus m\beta_r)_\n,
\end{eqnarray*}
where $\hat{\alpha}, \hat{\beta} \in Z(\mu_{m,r})$ and
$\alpha_0, \beta_0 \in Z(\mu_{m, d_0})$.

Let $\lambda=\lambda_{m,d_0}$.
By the induction hypothesis, $\alpha_0$ and $\lambda (\alpha_0)$
(resp.\ $\beta_0$ and $\lambda (\beta_0)$) are conjugate in $Z(\mu_{m, d_0})$.
By Lemmas~\ref{lem:choice} and \ref{lem:easy}~(ii),
we may assume that $\alpha_i$'s and $\beta_i$'s
are conjugacy representatives, that is,
$\alpha_i=\lambda(\alpha_i)$ and  $\beta_i=\lambda(\beta_i)$
for $0\le i\le r$.

\smallskip
The braids $\nu(\alpha)$ and $\nu(\beta)$ are written differently
according to whether $d_0\ge 1$ or $d_0=0$.
Let $\n'$ be a compositions of $n-1$ such that
$$
\begin{array}{lcll}
\n' &=& (md_0,\underbrace{d_1,\ldots,d_1}_m, \ldots,\underbrace{d_r,\ldots,d_r}_m)
& \mbox{if $d_0\ge 1$},\\
\n' &=& (\underbrace{d_1,\ldots,d_1}_m, \ldots,\underbrace{d_r,\ldots,d_r}_m)
& \mbox{if $d_0=0$}.
\end{array}
$$
By Corollary~\ref{cor:crs},
$\Rext(\nu(\alpha))$ and $\Rext(\nu(\beta))$
are obtained from $\Rext(\alpha)$ and $\Rext(\beta)$, respectively,
by forgetting the first puncture.
Since $\Rext(\alpha)=\Rext(\beta)=\C_\n$, one has
$$
\Rext(\nu(\alpha))=\Rext(\nu(\beta))=\C_{\n'}.
$$

If $d_0\ge 1$, then $\nu(\mu_{m,\d})$, $\nu(\alpha)$ and $\nu(\beta)$ are written as
\begin{eqnarray*}
\nu(\mu_{m,\d})
&=&\myangle{\mu_{m,r}}_{\n'} (\nu(\mu_{m,d_0})\oplus
{\textstyle \bigoplus_{i=1}^r}
(\underbrace{\Delta_{(d_i)}^2\oplus1\oplus\cdots\oplus1}_m) )_{\n'},\\
\nu(\alpha)
&=&\myangle{\hat\alpha}_{\n'} (\nu(\alpha_0)\oplus
m\alpha_1\oplus\cdots\oplus m\alpha_r)_{\n'},\\
\nu(\beta)
&=&\myangle{\hat\beta}_{\n'} (\nu(\beta_0)\oplus
m\beta_1\oplus\cdots\oplus m\beta_r)_{\n'}.
\end{eqnarray*}
In this case, $\nu(\alpha_0),\nu(\beta_0)\in Z(\nu(\mu_{m,d_0}))$.

If $d_0=0$, then $\nu(\mu_{m,\d})$, $\nu(\alpha)$ and $\nu(\beta)$ are written as
\begin{eqnarray*}
\nu(\mu_{m,\d})
&=&\myangle{\nu(\mu_{m,r})}_{\n'} (
{\textstyle \bigoplus_{i=1}^r}
(\underbrace{\Delta_{(d_i)}^2\oplus1\oplus\cdots\oplus1}_m) )_{\n'},\\
\nu(\alpha)
&=&\myangle{\nu(\hat\alpha)}_{\n'} (
m\alpha_1\oplus\cdots\oplus m\alpha_r)_{\n'}\\
\nu(\beta)
&=&\myangle{\nu(\hat\beta)}_{\n'} (
m\beta_1\oplus\cdots\oplus m\beta_r)_{\n'}.
\end{eqnarray*}
In this case, $\nu(\hat\alpha),\nu(\hat\beta)\in Z(\nu(\mu_{m,r}))$.

\begin{claim}{Claim 3}
$\hat\alpha$ and $\hat\beta$ are conjugate in $B_{mr+1}$.
\end{claim}

\begin{proof}[Proof of Claim 3]
If $\alpha$ and $\beta$ are conjugate in $B_n$, then
the $\n$-exterior braids
$\hat\alpha$ and $\hat\beta$ are conjugate in $B_{mr+1}$
by Lemma~\ref{lem:Rext}.

Suppose that $\nu(\alpha)$ and $\nu(\beta)$ are conjugate in $B_{n-1}$.
If $d_0\ge 1$, then $\hat\alpha$ and $\hat\beta$ are conjugate in $B_{mr+1}$
by Lemma~\ref{lem:Rext}
because $\Rext(\nu(\alpha))=\Rext(\nu(\beta))=\C_{\n'}$
and the $\n'$-exterior braids of $\nu(\alpha)$ and $\nu(\beta)$
are $\hat\alpha$ and $\hat\beta$.
Similarly, if $d_0=0$,
then $\nu(\hat\alpha)$ and $\nu(\hat\beta)$ are conjugate in $B_{mr}$,
hence $\hat\alpha$ and $\hat\beta$ are conjugate in $Z(\mu_{m,r})$
by the induction hypothesis.
\end{proof}

Since $\Rext(\alpha )=\C_{\n}$ and $\Ext_{\n}(\alpha)= \hat\alpha$,
the $\n$-exterior braid
$\hat\alpha$ is either periodic or pseudo-Anosov.

\case{Case 1}{$\hat{\alpha}$ is periodic.}
Let $k=mr+1$.
Since $\hat{\alpha}$ and $\hat{\beta}$ are pure and periodic,
they are central.
Because they are conjugate by the above claim, $\hat{\alpha}=\hat{\beta}=\Delta_{(k)}^{2u}$
for some integer $u$.

\begin{claim}{Claim 4}
$\alpha_0=\beta_0$ and there is an $r$-permutation $\theta$
such that $\beta_i=\alpha_{\theta(i)}$ for $1\le i\le r$.
\end{claim}

\begin{proof}[Proof of Claim 4]
If $\alpha$ and $\beta$ are conjugate in $B_n$, then
the following equality holds
between multisets consisting of conjugacy classes
by J.~Gonz\'alez-Meneses~\cite[Proposition 3.2]{Gon03}:
$$
\Bigl\{[\alpha_0], \underbrace{[\alpha_1],\ldots, [\alpha_1]}_m ,\ldots,
\underbrace{[\alpha_r],\ldots, [\alpha_r]}_m \Bigr\}
=\Bigl\{[\beta_0], \underbrace{[\beta_1],\ldots, [\beta_1]}_m ,\ldots,
\underbrace{[\beta_r],\ldots, [\beta_r]}_m \Bigr\}.
$$
Since $m\ge 2$,  $[\alpha_0]=[\beta_0]$
and there is an $r$-permutation $\theta$
such that $[\beta_i]=[\alpha_{\theta(i)}]$ for $1\le i\le r$.
By the induction hypothesis, $\alpha_0$ and $\beta_0$ are conjugate in $Z(\mu_{m,d_0})$.

If $\nu(\alpha)$ and $\nu(\beta)$ are conjugate in $B_{n-1}$ and $d_0\ge 1$,
then the following equality holds
between multisets consisting of conjugacy classes:
$$
\Bigl\{[\nu(\alpha_0)], \underbrace{[\alpha_1],\ldots, [\alpha_1]}_m ,\ldots,
\underbrace{[\alpha_r],\ldots, [\alpha_r]}_m \Bigr\}
=\Bigl\{[\nu(\beta_0)], \underbrace{[\beta_1],\ldots, [\beta_1]}_m ,\ldots,
\underbrace{[\beta_r],\ldots, [\beta_r]}_m \Bigr\}.
$$
Since $m\ge 2$, $[\nu(\alpha_0)]=[\nu(\beta_0)]$
and there is an $r$-permutation $\theta$
such that $[\beta_i]=[\alpha_{\theta(i)}]$ for $1\le i\le r$.
By the induction hypothesis, $\alpha_0$ and $\beta_0$ are conjugate in $Z(\mu_{m,d_0})$.

If $\nu(\alpha)$ and $\nu(\beta)$ are conjugate in $B_{n-1}$ and $d_0=0$,
then the following equality holds
between multisets consisting of conjugacy classes:
$$
\Bigl\{\underbrace{[\alpha_1],\ldots, [\alpha_1]}_m ,\ldots,
\underbrace{[\alpha_r],\ldots, [\alpha_r]}_m \Bigr\}
=\Bigl\{\underbrace{[\beta_1],\ldots, [\beta_1]}_m ,\ldots,
\underbrace{[\beta_r],\ldots, [\beta_r]}_m \Bigr\}.
$$
Then there is an $r$-permutation $\theta$
such that $[\beta_i]=[\alpha_{\theta(i)}]$ for $1\le i\le r$.
In this case, $\alpha_0$ and $\beta_0$ are the unique braid with one strand.

From the above three cases, we can see that
$\alpha_0$ is conjugate to $\beta_0$ in $Z(\mu_{m,d_0})$
and there is an $r$-permutation $\theta$ such that
$\beta_i$ is conjugate to $\alpha_{\theta(i)}$ for $1\le i\le r$.
Because we have assumed that each $\n$-interior braid of $\alpha$ and $\beta$
is the conjugacy representative with respect to $(m, d_0)$,
one has $\alpha_0=\beta_0$ and $\beta_i=\alpha_{\theta(i)}$ for $1\le i\le r$.
\end{proof}

In the above claim, $\theta$ must satisfy $d_i = d_{\theta(i)}$ for all $1\le i\le r$.
By the claim, we have
\begin{eqnarray*}
\alpha &=& \myangle{\Delta_{(k)}^{2u}}_\n (\alpha_0\oplus
m\alpha_1\oplus\cdots\oplus m\alpha_r)_\n, \\
\beta &=& \myangle{\Delta_{(k)}^{2u}}_\n (\alpha_0\oplus
m\alpha_{\theta(1)}\oplus\cdots\oplus m\alpha_{\theta(r)})_\n.
\end{eqnarray*}
By Lemma~\ref{lem:per},
there exists $\hat\zeta\in Z(\mu_{m,r})$ such that
the induced permutation of $\hat\zeta$ fixes $x_0$ and
sends $x_{i,j}$ to $x_{\theta(i), j}$
for $1\le i\le r$ and $1\le j\le m$ under the notation of the lemma.
Let $\zeta=\myangle{\hat\zeta}_\n$.
Then $\zeta\ast\C_{\n}=\C_{\n}$ because $d_i = d_{\theta(i)}$ for all $1\le i\le r$,
hence $\zeta\in Z(\mu_{m,\d})$ by Lemma~\ref{lem:cent}.
On the other hand,
\begin{eqnarray*}
\zeta\beta\zeta^{-1}
&=& \myangle{\hat\zeta}_\n
\myangle{\Delta_{(k)}^{2u}}_\n (\alpha_0\oplus
m\alpha_{\theta(1)} \oplus\cdots\oplus m\alpha_{\theta(r)})_\n
\myangle{\hat\zeta}_\n^{-1}\\
&=& \myangle{\hat\zeta}_\n
\myangle{\Delta_{(k)}^{2u}}_\n \myangle{\hat\zeta}_\n^{-1}
(\alpha_0\oplus
m\alpha_{1}\oplus\cdots\oplus m\alpha_{r})_\n\\
&=& \myangle{\Delta_{(k)}^{2u}}_\n (\alpha_0\oplus
m\alpha_{1}\oplus\cdots\oplus m\alpha_{r})_\n
=\alpha.
\end{eqnarray*}
Therefore $\alpha$ and $\beta$ are conjugate in $Z(\mu_{m,\d})$.

\case{Case 2}{$\hat{\alpha}$ is pseudo-Anosov.}

\begin{claim}{Claim 5}
The braids $\alpha$ and $\beta$ are conjugate in $B_n$.
\end{claim}

\begin{proof}[Proof of Claim 5]
If $\alpha$ and $\beta$ are conjugate in $B_n$, there is nothing to prove.
Therefore, assume that $\nu(\alpha)$ and $\nu(\beta)$ are conjugate in $B_{n-1}$.
Let $\gamma$ be an $(n-1)$-braid with $\nu(\beta)=\gamma^{-1}\nu(\alpha)\gamma$.
Since $\Rext(\nu(\alpha))=\Rext(\nu(\beta))=\C_{\n'}$,
one has $\gamma*\C_{\n'}=\C_{\n'}$ and
$\Ext_{\n'}(\nu(\beta))
=\Ext_{\n'}(\gamma)^{-1}\Ext_{\n'}(\nu(\alpha))\Ext_{\n'}(\gamma)$
by Lemma~\ref{lem:Rext}.

\smallskip

First, suppose that $d_0\ge 1$.
Because $\gamma*\C_{\n'}=\C_{\n'}$, $\gamma$ is of the form
$$\gamma
=\myangle{\hat\gamma}_{\n'} (\gamma_0\oplus
\underbrace{\gamma_{1,1}\oplus\cdots\oplus\gamma_{1,m}}_m\oplus\cdots\oplus
\underbrace{\gamma_{r,1}\oplus\cdots\oplus\gamma_{r,m}}_m)_{\n'}.
$$
In this case, $\hat\beta=\hat\gamma^{-1}\hat\alpha\hat\gamma$.
Because $\hat\alpha$ and $\hat\beta$ are pseudo-Anosov braids in $Z(\mu_{m,r})$,
one has $\hat\gamma\in Z(\mu_{m,r})$ by Proposition~\ref{prop:pAconj}.
In particular, $\hat\gamma$ is 1-pure by Lemma~\ref{lem:per}.
Since $\nu(\beta)=\gamma^{-1}\nu(\alpha)\gamma$ and
since $\hat\alpha$, $\hat\beta$ and $\hat\gamma$ are all 1-pure,
we have $\nu(\beta_0)=\gamma_0^{-1}\nu(\alpha_0)\gamma_0$.
By the induction hypothesis, there exists $\gamma_0'\in Z(\mu_{m,d_0})$ such that
$\beta_0=\gamma_0'^{-1}\alpha_0\gamma_0'$.
Let $\gamma'$ be the $n$-braid obtained from $\gamma$ by replacing
$\n'$ with $\n$ and $\gamma_0$ with $\gamma_0'$, that is,
$$\gamma'
=\myangle{\hat\gamma}_{\n} (\gamma_0'\oplus
\underbrace{\gamma_{1,1}\oplus\cdots\oplus\gamma_{1,m}}_m\oplus\cdots\oplus
\underbrace{\gamma_{r,1}\oplus\cdots\oplus\gamma_{r,m}}_m)_{\n}.
$$
We can easily show $\beta=\gamma'^{-1}\alpha\gamma'$
by using the fact that $\nu(\beta)=\gamma^{-1}\nu(\alpha)\gamma$,
$\hat\gamma$ is 1-pure, and
$\beta_0=\gamma_0'^{-1}\alpha_0\gamma_0'$.

\smallskip
Now, suppose that $d_0=0$.
Because $\gamma*\C_{\n'}=\C_{\n'}$, $\gamma$ is of the form
$$
\gamma
=\myangle{\hat\gamma}_{\n'} (
\underbrace{\gamma_{1,1}\oplus\cdots\oplus\gamma_{1,m}}_m\oplus\cdots\oplus
\underbrace{\gamma_{r,1}\oplus\cdots\oplus\gamma_{r,m}}_m)_{\n'}.
$$
In this case, $\nu(\hat \beta)=\hat\gamma^{-1}\nu(\hat\alpha)\hat\gamma$.
Since $\hat\alpha$ and $\hat\beta$ are pseudo-Anosov braids in $Z(\mu_{m,r})$,
$\nu(\hat \alpha)$ and $\nu(\hat \beta)$ are pseudo-Anosov
braids in $Z(\nu(\mu_{m,r}))$ by Corollary~\ref{cor:dyn}.
Thus $\hat\gamma\in Z(\nu(\mu_{m,r}))$ by Proposition~\ref{prop:pAconj}.
Because $\nu : Z(\mu_{m,r})\to Z(\nu(\mu_{m,r}))$ is an isomorphism,
there exists $\hat\gamma'\in Z(\mu_{m,r})$ with $\nu(\hat\gamma')=\hat\gamma$.
Then $\hat \beta=\hat\gamma'^{-1}\hat\alpha\hat\gamma'$.
In addition, $\hat\gamma'$ is 1-pure by Lemma~\ref{lem:per}.
Since $d_0=0$, both $\alpha_0$ and $\beta_0$ are the unique braid with
one strand.
Let $\gamma'$ be the $n$-braid obtained from $\gamma$ by replacing
$\n'$ with $\n$ and $\hat\gamma$ with $\hat\gamma'$ and by inserting
the trivial 1-braid as follows:
$$
\gamma'
=\myangle{\hat\gamma'}_{\n} (1\oplus
\underbrace{\gamma_{1,1}\oplus\cdots\oplus\gamma_{1,m}}_m\oplus\cdots\oplus
\underbrace{\gamma_{r,1}\oplus\cdots\oplus\gamma_{r,m}}_m)_{\n}.
$$
We can easily show $\beta=\gamma'^{-1}\alpha\gamma'$
by using the fact that $\nu(\beta)=\gamma^{-1}\nu(\alpha)\gamma$,
$\hat \beta=\hat\gamma'^{-1}\hat\alpha\hat\gamma'$,
and $\hat\gamma'$ is 1-pure.
\end{proof}

Recall that $\alpha$ and $\beta$ are of the form
\begin{eqnarray*}
\alpha
&=& \myangle{\hat{\alpha}}_\n (\alpha_0\oplus m\alpha_1\oplus\cdots\oplus m\alpha_r)_\n,\\
\beta
&=& \myangle{\hat{\beta}}_\n(\beta_0\oplus m\beta_1\oplus\cdots\oplus m\beta_r)_\n,
\end{eqnarray*}
where $\hat{\alpha}, \hat{\beta} \in Z(\mu_{m,r})$,
$\alpha_0, \beta_0 \in Z(\mu_{m, d_0})$,
$\alpha_i = \lambda(\alpha_i)$ and $\beta_i = \lambda(\beta_i)$
for $0\le i\le r$.

Because $\alpha$ and $\beta$ are conjugate in $B_n$ by the above claim,
there is an $n$-braid $\gamma$ such that
$$
\beta=\gamma\alpha\gamma^{-1}.
$$
Since $\Rext(\alpha)=\Rext(\beta)=\C_\n$, we have $\gamma*\C_\n=\C_\n$,
hence $\gamma$ is of the form
$$
\gamma
=\myangle{\hat\gamma}_{\n} (\gamma_0\oplus
\underbrace{\gamma_{1,1}\oplus\cdots\oplus\gamma_{1,m}}_m\oplus\cdots\oplus
\underbrace{\gamma_{r,1}\oplus\cdots\oplus\gamma_{r,m}}_m)_{\n}.
$$
Since $\hat{\alpha}$ and $\hat{\beta}$ are pseudo-Anosov braids in $Z(\mu_{m,r})$
and $\hat{\beta}=\Ext_\n(\beta)=\Ext_\n(\gamma\alpha\gamma^{-1})
=\hat{\gamma}\hat{\alpha}\hat{\gamma}^{-1}$, we have
$\hat{\gamma}\in Z(\mu_{m,r})$ by Proposition~\ref{prop:pAconj}.
Let $\theta$ be the $(mr+1)$-permutation induced by $\hat\gamma$.
By Lemma~\ref{lem:per}, $\theta(x_0)=x_0$ and
there exist integers $1\le k_i\le r$
and $0\le \ell_i<m$ such that
$\theta(x_{i,j})=x_{k_i, j-\ell_i}$ for $1\le i\le r$ and $1\le j\le m$.
Define
$$
\chi = \myangle{\hat{\gamma}}_\n
( 1\oplus
{\textstyle\bigoplus_{i=1}^r}
(\underbrace{\Delta_{(d_i)}^2\oplus\cdots\oplus\Delta_{(d_i)}^2}_{\ell_i}
\oplus \underbrace{1\oplus\cdots\oplus 1}_{m-\ell_i}) )_\n.$$
Since $\chi\ast\C_{\n}
=\myangle{\hat\gamma}_\n*\C_\n
=\gamma*\C_\n=\C_{\n}$,
one has $d_i=d_{k_i}$ for $1\le i\le r$.
Thus $\chi\in Z(\mu_{m,\d})$ by Lemma~\ref{lem:cent}.
Let
$$
\beta'=\chi^{-1}\beta\chi.
$$
Since $\beta,\chi\in Z(\mu_{m,\d})$ with $\beta*\C_\n=\chi*\C_\n=\C_\n$,
one has $\beta'\in Z(\mu_{m,\d})$ with $\beta'*\C_\n =\C_\n$.
Since $\Ext_\n(\beta') =\Ext_\n(\chi^{-1} \beta \chi)
=\hat\gamma^{-1}\hat\beta\hat\gamma=\hat\alpha$ is pure,
$\beta'$ is of the following form by Corollary~\ref{cor:cent}:
$$
\beta' = \myangle{\hat{\alpha}}_\n(\beta_0'\oplus
m\beta_1'\oplus\cdots\oplus m\beta_r')_\n,
\qquad
\beta_0' \in Z(\mu_{m, d_0}).
$$

Since both $\Ext_\n(\beta)$ and $\Ext_\n(\chi)$ are 1-pure,
one has $\beta_0'=\beta_0$, hence
$\beta_0'=\beta_0=\lambda(\beta_0)=\lambda(\beta_0')$.
By Lemma~\ref{lem:choice}, there exists an $\n$-split braid $\chi'$ in $Z(\mu_{m, \d})$
such that
\begin{equation}\label{eq:1}
\beta''
=\chi'^{-1}\beta'\chi'
= \myangle{\hat{\alpha}}_\n(\beta_0''\oplus
m\beta_1''\oplus\cdots\oplus m\beta_r'')_\n,
\end{equation}
where $\beta_i''$'s are conjugacy representatives,
that is, $\beta_i''=\lambda(\beta_i'')$ for $0\le i\le r$.

\smallskip
Now we will show that $\beta''=\alpha$.
Let $\chi''=\gamma^{-1}\chi$.
Then $\chi''*\C_\n=\C_\n$.
Because $\Ext_\n(\chi'')=\Ext_\n(\gamma)^{-1}\Ext_\n(\chi)=\hat\gamma^{-1}\hat\gamma=1$,
$\chi''$ is an $\n$-split braid.
Note that
$$\beta''=\chi'^{-1}\beta'\chi'
=\chi'^{-1}\chi^{-1}\beta\chi\chi'
=\chi'^{-1}\chi^{-1}\gamma\alpha\gamma^{-1}\chi\chi'
=(\chi''\chi')^{-1}\alpha(\chi''\chi').
$$
Because both $\chi'$ and $\chi''$ are $\n$-split braids, so is $\chi''\chi'$.
Therefore $\chi''\chi'$ is of the form
$$
\chi''\chi' =
\myangle{1}_\n
(\chi_0\oplus {\textstyle\bigoplus_{i=1}^r}
(\chi_{i,1}\oplus\cdots\oplus\chi_{i,m}) )_\n.
$$
Because the $\n$-exterior braids of $\alpha$ and $\chi''\chi'$ are
pure braids, the conjugation $(\chi''\chi')^{-1}\alpha(\chi''\chi')$
can be computed component-wise:
\begin{equation}\label{eq:2}
(\chi''\chi')^{-1}\alpha(\chi''\chi')=
\myangle{\hat\alpha}_\n
(\chi_0^{-1}\alpha_0\chi_0\oplus
{\textstyle\bigoplus_{i=1}^r}
(\chi_{i,1}^{-1}\alpha_i\chi_{i,1}\oplus\cdots\oplus
\chi_{i,m}^{-1}\alpha_i\chi_{i,m}) )_\n.
\end{equation}
Comparing $(\ref{eq:1})$ and $(\ref{eq:2})$, we can see that
$\alpha_i$ and $\beta_i''$ are conjugate for $0\le i\le r$.
Because $\alpha_i$ and $\beta_i''$ are conjugacy representatives
with respect to $(m, d_0)$,
we have $\alpha_i=\beta_i''$ for $0\le i\le r$, hence $\alpha=\beta''$.
Because
$$
\alpha=\beta''=(\chi\chi')^{-1}\beta(\chi\chi')
$$
and $\chi,\chi'\in Z(\mu_{m,\d})$, $\alpha$ and $\beta$ are conjugate
in $Z(\mu_{m,\d})$.
\end{proof}

We remark that in the above proof
we can show by a straightforward computation that
$\beta_i'=\beta_{k_i}$ for $i=1,\ldots,r$.
In particular, the $\n$-interior braids $\beta_0',\ldots,\beta_r'$
are conjugacy representatives with respect to $(m,d_0)$,
hence we can take $\chi'=1$ and $\beta''=\beta'$.
This means that $\alpha=\chi^{-1}\beta\chi$,
where $\chi$ depends only on $\n$ and $\hat\gamma$.
However we introduced $\chi'$ and $\beta''$
because it is much simpler than showing $\beta_i'=\beta_{k_i}$.

\subsection*{Acknowledgements}

The authors are grateful to the anonymous referee
for valuable comments and suggestions.
This work was done partially while the authors were visiting
Department of Mathematical and Computing Sciences,
Tokyo Institute of Technology in February 2008.
We thank Sadayoshi Kojima and Eiko Kin
for their hospitality and interest in this work.
This work was supported by the Korea Research Foundation Grant funded
by the Korean Government (MOEHRD, Basic Research Promotion Fund)
(KRF-2008-331-C00039).

\end{document}